\documentclass[12pt]{amsart}
\usepackage{setspace}
\usepackage{verbatim}

\usepackage{amsfonts,fancyhdr,ifthen,bm}
\usepackage{amscd,amssymb} 
\usepackage[letterpaper,left=1.35in,right=1.35in,top=1.35in,bottom=1.35in]{geometry}
\usepackage{times}
\usepackage{graphicx}
\usepackage{color}
\usepackage{epsfig}
\usepackage{mathrsfs}
\usepackage{paralist}
\usepackage{comment}
\usepackage{mathtools}
\usepackage{multirow}
\usepackage{tikz-cd}
\usepackage[foot]{amsaddr}

\usepackage{enumitem}

\usepackage{appendix}

\usepackage{tikz,ifthen,calc}
\usetikzlibrary{automata,arrows,positioning,calc}

\DeclareMathOperator{\vv}{\mathbf{v}}
\DeclareMathOperator{\ww}{\mathbf{w}}

\newtheorem{thm}{Theorem}[section]

\newtheorem*{thm*}{Theorem}
\newtheorem{prop}[thm]{Proposition}
\newtheorem*{prop*}{Proposition}
\newtheorem{cor}[thm]{Corollary}
\newtheorem*{cor*}{Corollary}

\newtheorem{lem}[thm]{Lemma}
\newtheorem*{lem*}{Lemma}

\newtheorem*{oquest*}{Open Question}

\newtheorem*{ssstproblem}{Schur--Siegel-Smyth Trace Problem}

\theoremstyle{remark}
\newtheorem{rmk}[thm]{Remark}

\theoremstyle{remark}
\newtheorem*{rmk*}{Remark}

\theoremstyle{definition}
\newtheorem{defn}[thm]{Definition}

\theoremstyle{definition}
\newtheorem{notat}[thm]{Notation}

\theoremstyle{definition}

\theoremstyle{definition}

\theoremstyle{definition}
\newtheorem*{defn*}{Definition}

\theoremstyle{definition}
\newtheorem{ex}[thm]{Example}

\numberwithin{equation}{section}%numbers equations by section

\newcommand{\C}{\mathbb{C}}

\newcommand{\QQ}{\mathbb{Q}}

\newcommand{\R}{\mathbb{R}}

\DeclareMathOperator{\FFF}{\mathbb{F}}

\newcommand{\normf}[1]{\left|\left|w_{\mu}^{n} #1\right|\right|_{\Sigma}}
\newcommand{\normz}[1]{\left|\left| #1\right|\right|_{\Sigma}}
\newcommand{\normff}[2]{\left|\left|w_{\mu}^{#2} #1\right|\right|_{\Sigma}}

\newcommand{\res}{\textup{res}}

\DeclareFontFamily{U}{wncy}{}
\DeclareFontShape{U}{wncy}{m}{n}{<->wncyr10}{}
\DeclareSymbolFont{mcy}{U}{wncy}{m}{n}
\DeclareMathSymbol{\Sha}{\mathord}{mcy}{"58}

\newcommand\restr[2]{{% we make the whole thing an ordinary symbol
  \left.\kern-\nulldelimiterspace % automatically resize the bar with \right
  #1 % the function
  \vphantom{\big|} % pretend it's a little taller at normal size
  \right|_{#2} % this is the delimiter
  }}

\DeclareMathOperator*{\Vol}{Vol}

\DeclareMathOperator*{\tr}{tr}
\DeclareMathOperator*{\disc}{disc}

\newcommand{\weaks}{$\text{weak}^*\,$}

\begin{document}
\title{Algebraic integers with conjugates in a prescribed distribution}

\author{Alexander Smith}
\email{asmith13@math.ucla.edu}
\date{\today}

\maketitle

\begin{abstract}
Given a compact subset $\Sigma$ of $\R$ obeying some technical conditions, we consider the set of algebraic integers whose conjugates all lie in $\Sigma$. The distribution of conjugates of such an integer defines a probability measure on $\Sigma$; our main result gives a necessary and sufficient condition for a given probability measure on $\Sigma$ to be the limit of some sequence of distributions of conjugates. As one consequence, we show there are infinitely many totally positive algebraic integers $\alpha$ with $\tr(\alpha) < 1.89831 \cdot\deg(\alpha)$. We also show how this work can be applied to find simple abelian varieties over finite fields with extreme point counts.
\end{abstract}

\section{Introduction}
\label{sec:intro}

Given an algebraic integer $\alpha \in \C$,  take $\alpha_1 = \alpha,\, \alpha_2, \dots, \,\alpha_n$ to be the complex roots of the minimal polynomial of $\alpha$ over $\QQ$. The degree and trace of $\alpha$ are then defined by
\[\deg(\alpha) = n \,\,\,\text{ and }\,\, \tr(\alpha) = \alpha_1 + \alpha_2 + \dots + \alpha_n.\]
We call $\alpha$ totally positive if the conjugates $\alpha_1, \dots, \alpha_n$ are all positive real numbers. Take $\lambda_{\textup{SSS}}$ to be the least real number such that, for any $\epsilon > 0$, there are only finitely many totally positive algebraic integers $\alpha$ satisfying $\tr(\alpha) < (\lambda_{\textup{SSS}} - \epsilon) \deg (\alpha)$.

In 1918, Schur proved \cite[Satz XI]{Schur18} that $\lambda_{\textup{SSS}}$ satisfies
\[e^{1/2}  \le \lambda_{\textup{SSS}} \le 2.\]
In 1945, Siegel improved the lower bound in this inequality to $1.7336\dots$ \cite[Theorem II]{Sieg45}. In 1984, Smyth improved it to $1.7719$ \cite{Smyth84b}. Since then, the method Smyth introduced in \cite{Smyth84a, Smyth84b} has been used repeatedly to improve the lower bound \cite{McSmyth04, ABP06, AgPe07, AgPe08, Flam09, McKee11, Liang11}. The current state of the art appears in \cite{WaWuWu21}, and takes the form
\[1.793145 \le \lambda_{\textup{SSS}} \le 2.\]
This work fits into the framework of the following problem, which was codified by Borwein in \cite{Borw02}.
\begin{ssstproblem}
Fix $\epsilon > 0$. Show that there are only finitely many totally positive algebraic integers satisfying 
\begin{equation}
\label{eq:trace_problem}
\tr(\alpha) \le (2 - \epsilon)\deg(\alpha),
\end{equation}
 and explicitly compute this list of exceptions if possible.
\end{ssstproblem}
A detailed account of the history of this problem can be found in \cite{AgPe08}. This problem statement reflects the general consensus that $\lambda_{\textup{SSS}}$ should equal $2$. Our first result is that this is not the case.
\begin{thm}
\label{thm:Serre}
We have $\lambda_{\textup{SSS}} < 1.89831$.
\end{thm}
The constant $1.89831$ hints at the method underlying our work. To explain this, we need to recall the approach to the trace problem pioneered by Smyth in \cite{Smyth84a, Smyth84b}. For a given $\lambda > 0$, suppose that one has found a finite sequence $a_1, \dots, a_N$ of positive numbers and a finite sequence $Q_1, \dots, Q_N$ of nonzero integer polynomials so
\begin{equation}
\label{eq:Smyth}
x \ge \lambda + \sum_{k =1}^N a_k \log|Q_k(x)|\quad\text{for all } x > 0.
\end{equation}
Given a totally positive algebraic integer $\alpha$ with conjugates $\alpha_1 = \alpha,\, \alpha_2, \dots, \,\alpha_n$, we may sum this inequality over the $\alpha_i$. This yields
\[\sum_{i \le n} \alpha_i \ge \lambda n + \sum_{k \le N}a_k \log\left|\prod_{i \le n} Q_k(\alpha_i)\right|.\]
The product $\prod_{i \le n} Q_k(\alpha_i)$ is recognizable as the resultant $\res(P, Q_k)$, where $P$ is the minimal polynomial of $\alpha$. In particular, if $\alpha$ is not a root of some $Q_k$, $\res(P, Q_k)$ is a nonzero integer for each $k$, and we are left with
\[\tr(\alpha) \ge \lambda \deg(\alpha).\]
Smyth's original article gives an instance of \eqref{eq:Smyth} with $\lambda = 1.7719$ and about $15$ auxiliary polynomials. The current state of the art \cite{WaWuWu21} gives an instance of \eqref{eq:Smyth} with $\lambda = 1.793145$ and $130$ auxiliary polynomials.

It was observed by Smyth \cite{Smyth99} and Serre \cite[Appendix B]{AgPe08} that there were values of $\lambda$ less than $2$ for which \eqref{eq:Smyth} would not hold for any choice of $Q_1, \dots, Q_N$ and $a_1, \dots, a_N$. Smyth wrote that this suggested ``that perhaps $2$ is not in fact the smallest limit point [of the ratios $\tr(\alpha)/\deg(\alpha)$]''  \cite[p. 316]{Smyth99}, but these results are more often viewed as evidence of the limitations of Smyth's method. 

But Smyth's optimistic interpretation is correct. As we will show in Theorem \ref{thm:general_Smyth}, $\lambda_{\textup{SSS}}$ is the least upper bound of the $\lambda$ for which there is some $N \ge 1$, some sequence of nonzero integer polynomials $Q_1, \dots, Q_N$, and some positive real numbers $a_1, \dots, a_N$ so \eqref{eq:Smyth} holds. As we discuss in detail in Example \ref{ex:Serre}, Serre proved that no such inequality can hold for $\lambda > 1.89830\dots$, so Theorem \ref{thm:general_Smyth} suffices to prove Theorem \ref{thm:Serre}.

\begin{rmk}
The constant $1.89831$ is not optimal for two reasons. First, our work in Example  \ref{ex:Serre} uses rounded forms of the optimal parameters for Serre's example. More fundamentally, even with perfectly optimized parameters, Serre's argument cannot produce the actual value for $\lambda_{\text{SSS}}$. See Proposition  \ref{prop:crater} for more details.
\end{rmk}

Our work has applications outside the Schur--Siegel--Smyth trace problem. A notable application is to the question of extreme point counts for abelian varieties over finite fields, which reduces by Honda--Tate theory \cite{Honda68} to an optimization question akin to the trace problem; see Proposition \ref{prop:Honda}. One consequence of this is the following result, which follows from Theorem \ref{thm:Serre} by Proposition \ref{prop:Kade}.
\begin{cor}
\label{cor:Serre2}
There is some $C > 0$ so we have the following:

Choose a square prime power $q$ no smaller than $C$. Then there are infinitely many $\FFF_q$-simple abelian varieties $A$ over $\FFF_q$ satisfying
\[\#A(\FFF_q) \,\ge\, \left(q + 2\sqrt{q} - 0.89831\right)^{\dim A}\]
and infinitely many more satisfying
\[\#A(\FFF_q)\,\le\, \left(q - 2\sqrt{q} + 2.89831\right)^{\dim A}.\]
\end{cor}
This improves on the prior records for this problem found in \cite{Kade21} and \cite{BCLPS21}. Further optimization work could improve this result and extend it to non-square $q$. The constants replacing $0.89831$ and $2.89831$ in such a sharpened generalization would depend on the fractional part of $2\sqrt{q}$.

Theorem \ref{thm:general_Smyth} is a consequence of the main theorem of this paper, Theorem \ref{thm:main}. Subject to some restrictions, this theorem characterizes the measures on $\R$ that are attainable as the limit of the distribution of conjugates of some sequence of totally real algebraic integers $\alpha_1, \alpha_2, \dots$, in the sense considered by Serre in \cite{Serre18}. In line with our vindication of Smyth's method,
we find that the only obstruction to obtaining a given measure is the integrality of the resultant of integer polynomials. 

We introduce some notation to make this precise.
\begin{defn}
\label{defn:counting}
Throughout this paper, the term \emph{Borel measure}, or just \emph{measure}, will denote a finite positive measure on the $\sigma$-algebra of Borel sets of $\C$.

For any complex number $\alpha$, take $\delta_{\alpha}$ to be the Borel measure defined by 
\[\delta_{\alpha}(Y) = \begin{cases} 1 &\text{ if } \alpha \in Y \\ 0 & \text{ otherwise.}\end{cases}\]
Given a complex polynomial $P(z) = a_n \cdot \prod_{i \le n} (z - \alpha_i)$ of degree $n \ge 1$, we follow \cite{Smyth84a} and \cite[(1.2.1)]{Serre18} to define the associated counting measure $\mu_P$ by 
\[\mu_P = \tfrac{1}{n}( \delta_{\alpha_1} + \dots + \delta_{\alpha_n}).\]
This is a probability measure supported on $\{\alpha_1, \dots, \alpha_n\}$.

Given a compact subset $\Sigma$ of $\C$, we will endow the set of Borel measures whose support is contained in $\Sigma$ with the \weaks topology, where an infinite sequence $\mu_1, \mu_2, \dots$ converges to $\mu$ if and only if
\[\lim_{k \to \infty} \int_{\Sigma} f d\mu_k = \int_{\Sigma} f d\mu\]
for every continuous function $f: \Sigma \to \R$.
\end{defn}

\begin{thm}
\label{thm:main}
Take $\Sigma$ to be a compact subset of $\R$ with at most countably many components. We assume that $\Sigma$ has capacity strictly larger than $1$ (see Definition \ref{defn:capacity}).

Then, for any Borel probability measure $\mu$ with support contained in $\Sigma$, the following two conditions are equivalent:
\begin{enumerate}
\item For every nonzero integer polynomial $Q$,
\[\int_{\Sigma} \log|Q(x)| d\mu(x) \ge 0.\]
\item There is an infinite sequence of distinct irreducible monic integer polynomials $R_1, R_2, \dots$ such that the support of $\mu_{R_k}$ is contained in $\Sigma$ for every $k$ and such that $\mu_{R_1}, \mu_{R_2}, \dots$ has \weaks limit $\mu$.
\end{enumerate}
\end{thm}

The proof that condition (2) implies condition (1) in this theorem was given by Serre as \cite[Lemma 1.3.4]{Serre18}. Given a nonzero integer polynomial $Q$, Serre notes that $\log|Q|$ is upper semicontinuous on $\Sigma$, and concludes that
\[\int_{\Sigma} \log|Q| d\mu \ge \limsup_{k \to \infty} \int_{\Sigma} \log|Q| d\mu_{R_k} =  \limsup_{k \to \infty} \frac{\log |\res(R_k, Q)|}{\deg R_k} \ge 0,\]
with the first inequality following from the monotone convergence theorem.

The proof of the converse, that (1) implies (2), is new and accounts for most of the length of this paper. First, using Proposition \ref{prop:limit_Holder}, we are able to assume without loss of generality that $\mu$ and $\Sigma$ satisfy some niceness properties. Under this assumption, we use the coefficient adjustment technique of Robinson \cite{Robi64} and Fekete and Szeg\H{o} \cite[Theorem D]{FeSz55} to  construct monic real polynomials with some integral coefficients whose associated counting measures approximate $\mu$. We can then apply the geometry of numbers to adjust the remaining non-integral coefficients to integers.

After some preliminary work in Section \ref{sec:prelim}, we present the results we need from the geometry of numbers in Section \ref{sec:GON} and give our coefficient adjustment argument in Section \ref{sec:adjust}. The proof of Proposition \ref{prop:limit_Holder} is given in Section \ref{sec:limits}, finishing the proof that (1) implies (2). After finishing  the proof, we will show that it implies the other results mentioned so far.

It would be interesting to remove some of the restrictions placed on $\Sigma$ in Theorem \ref{thm:main}. It is straightforward to find totally disconnected compact subsets of $\R$ of capacity greater than $1$ that contain no algebraic number, so the restriction to $\Sigma$ with countably many components is probably inevitable. The restriction to subsets of $\R$ avoids annoying complex sets like $\{z \in \C\,:\,\,|z| = 3/2\}$, but it seems reasonable to expect the theorem to hold for e.g$.$ closures of open sets in $\C$, with \cite{FeSz55} making some progress in this direction. We note that the study of Weil numbers reduces to the study of totally real algebraic integers whose conjugates lie in certain intervals, so a version of Theorem \ref{thm:main} for complex $\Sigma$ is not necessary to study such numbers.

Removing the restriction to $\Sigma$ of capacity strictly greater than $1$ would be of greatest interest. From Proposition \ref{prop:general_energy}, we know this would only be interesting for $\Sigma$ of capacity exactly $1$, and that the only measure to consider would be the unweighted equilibrium measure. The second condition of Theorem \ref{thm:main} is known for some special $(\Sigma, \mu)$ of this form \cite{Robi64}, but it remains unknown whether this condition holds more generally.

\subsection*{Acknowledgments}
We would like to thank Jean--Pierre Serre and Chris Smyth for their comments on this paper and for providing their relevant correspondence.

We would also like to thank Frank Calegari, Brian Conrad, Pol van Hoften,  Borys Kadets, Wanlin Li, and Bjorn Poonen for useful feedback on this project, and the anonymous referees for detailed comments on previous drafts of this paper.

This research was partially conducted during the period the author served as a Clay Research Fellow.

\section{Preliminaries on measures and polynomials}
\label{sec:prelim}
\begin{defn}
\label{defn:capacity}
Choose a compact subset $\Sigma$ of $\C$. Given a Borel measure $\mu$ supported on $\Sigma$, we define the potential function $U^{\mu}: \C \to \R \cup\{\infty\}$ of $\mu$ by
\[U^{\mu}(z) = \int_{\Sigma} -\log|z - w| d\mu(w),\]
and we define the energy of $\mu$ by
\[I(\mu) = \int_{\Sigma} U^{\mu}(z)d\mu(z) = \int_{\Sigma} \int_{\Sigma} -\log|z - w| d\mu(w) d\mu(z).\]
The energy of $\mu$ lies in $\R\cup \{\infty\}$. Following e.g. \cite{Ransford95}, we define the capacity of $\Sigma$ by
\[\kappa_{\Sigma} \,=\, {\sup}_{\mu}\, e^{ -I(\mu)},\]
 where the supremum is taken over all probability measures with support contained in $\Sigma$. If $\Sigma$ has positive capacity, the supremum is attained for a unique probability measure of minimal energy known as the unweighted equilibrium measure of $\Sigma$.
\end{defn}

For most of our proof of Theorem \ref{thm:main}, we will restrict our attention to measures that have relatively nice potential functions.
\begin{defn}
\label{defn:Holder}
Given a Borel measure $\mu$ supported on $\R$ and $\eta > 0$, we say that $\mu$ is  \emph{$\eta$-H\"{o}lder} if there is $C> 0$ for which
\begin{equation}
\label{eq:Holder_condition}
\mu([x, y]) \le C \cdot (y - x)^{\eta} \quad\text{ for all } x, y \in \R\text{ with } x < y.
\end{equation}
We write that $\mu$ is \emph{H\"{o}lder} if it is $\eta$-H\"{o}lder for some $\eta > 0$.
\end{defn}

Given a Borel measure $\mu$, real numbers $x$ and $b$ with $b > 0$, and a subinterval $I$ of $[x-b, x+b]$, we observe that
\begin{align}
\label{eq:Tonelli}
\int_I \log|x -t| d\mu(t) &= \,\int_I\left( \log b -  \int_{|x-t|}^b r^{-1} dr \right) d\mu(t) \\
&=\, \log b \cdot \mu(I) - \int_{0}^b r^{-1}  \cdot \mu\left([x- r, \, x+ r] \cap I\right)dr \nonumber
\end{align}
from Tonelli's theorem \cite[Theorem 14.2]{DiBe16}. This observation combines well with the H\"{o}lder condition, as we see in the proof of the following lemma.
\begin{lem}
\label{lem:pot_Hold}
Given $\eta$ in $(0, 1)$ and an $\eta$-H\"{o}lder measure $\mu$ supported in a compact subset of $\R$, the potential $U^{\mu}$ is a finite, $\eta$-H\"{o}lder continuous function on $\R$. That is to say, there is a positive number $C$ so 
\begin{equation}
\label{eq:pot_Hold}
\left|U^{\mu}(y) - U^{\mu}(x) \right|\,\le\, C \cdot |y - x|^{\eta}  \quad\text{ for all } x, y \in \R.
\end{equation}
\end{lem}
\begin{proof}
Choose $C_0 > 0$ so that \eqref{eq:Holder_condition} holds for the measure $\mu$ with $C = C_0$.  Choose $x, y \in \R$ and take $m = |y-x|/2$. The statement is clear if $x = y$, so we suppose $x \ne y$. Define $I_1$ to be the interval from $x - m$ to $x + m$. If $x < y$, define $I_2$ to be the interval $(-\infty, x - m]$; otherwise, take $I_2$ to be the interval $[x + m, \infty)$.

We have
\[\int_{I_1} \log\frac{|y -t|}{|x-t|} d\mu(t) \le \mu(I_1) \log 3m - \int_{I_1} \log|x-t|d\mu(t).\]
We have $\mu(I_1) \le C_0 (2m)^{\eta}$ by \eqref{eq:Holder_condition}. Applying \eqref{eq:Tonelli} with $b = m$ and the H\"{o}lder condition to the final integral, we find that this right hand side is no more than
\[C_0 (2m)^{\eta} \cdot (\log 3m - \log m) + \int_0^{m} r^{-1} C_0(2r)^{\eta}dr \le C_0 \cdot (\log 3 + \eta^{-1}) \cdot |x - y|^\eta.\]
Meanwhile, the definition of $I_2$ gives 
\[\mu([y - r, \,y+r] \cap I_2) = \begin{cases} 0 &\text{ if } r < 3m \\ \mu([x -r + 2m,\, x+ r -2m] \cap I_2) &\text{ if } r \ge 3m,\end{cases}\]
so \eqref{eq:Tonelli} gives
\[\int_{I_2} \log|y-t| d\mu(t) \,= \,\mu(I_2) \log b  \,-\, \int_m^{b - 2m} (r + 2m)^{-1} \cdot\mu([x - r,x+r] \cap I_2)dr\]
 for $b$ sufficiently large. Applying \eqref{eq:Tonelli} at $x$, taking a difference with this last equation, and letting $b$ tend to infinity gives
\[\int_{I_2} \log\frac{|y -t|}{|x-t|} d\mu(t)  = \int_{m}^{\infty} \left(r^{-1} - (r+ 2m)^{-1}\right) \mu\left( [x - r, x + r] \cap I_2\right)dr.\]
We have $r^{-1} - (r + 2m)^{-1} \le 2m r^{-2}$ for $r \ge 0$, so this is at most
\[\int_m^{\infty} 2mr^{-2} \cdot C_0(2r)^{\eta}dr =2C_0(1- \eta)^{-1} \cdot |x - y|^\eta.\]
Finally, we have
\[\int_{\R \backslash (I_1 \cup I_2)} \log\frac{|y -t|}{|x-t|} d\mu(t) \le 0\]
since $|y- t|$ is less than $|x - t|$ on $\R \backslash (I_1 \cup I_2)$. In short, we have
\[U^{\mu}(x) \le U^{\mu}(y) +C_0\left(\eta^{-1}+  2(1 - \eta)^{-1} + \log 3\right) \cdot |x - y|^{\eta}\quad\text{for all } x, y \in \R.\]
Since $U^{\mu}(y)$ is finite for some $y$, we may conclude from this inequality that $U^{\mu}(x)$ is always finite and that \eqref{eq:pot_Hold} is satisfied with $C = C_0\left(\eta^{-1}+  2(1 - \eta)^{-1} + \log 3\right)$.
\end{proof}

Much of our work also requires a slightly nicer class of compact subsets $\Sigma$.
\begin{defn}
\label{defn:CFUOI}
Given a compact subset $\Sigma$ of $\R$, we call $\Sigma$ a \emph{compact finite union of intervals} if $\Sigma$ is nonempty, contains finitely many connected components, and has no isolated points.
\end{defn}

The following proposition shows there is no issue in restricting our proof of Theorem \ref{thm:main} to H\"{o}lder measures on compact finite unions of intervals.
\begin{prop}
\label{prop:limit_Holder}
Suppose $\Sigma$ is a compact subset of $\R$ with at most countably many components and capacity strictly greater than $1$, and choose a probability measure $\mu$ with support contained in $\Sigma$ satisfying condition (1) of Theorem \ref{thm:main}. Then there is a sequence $\mu_1, \mu_2, \dots$ of H\"{o}lder probability measures and a sequence $\Sigma_1, \Sigma_2, \dots$ of closed subsets of $\Sigma$ such that
\begin{enumerate}
\item For every $k \ge 1$, $\Sigma_k$ is a compact finite union of intervals, and $\mu_k$ has support contained in $\Sigma_k$;
\item For every $k \ge 1$ and nonzero integer polynomial $Q$, $\int_{\Sigma_k} \log|Q| d\mu_k \ge 0$; and
\item The sequence $\mu_1, \mu_2, \dots$ \weaks converges to $\mu$.
\end{enumerate}
\end{prop}
We will prove this proposition in Section \ref{sec:limits}.

\begin{defn}
\label{defn:norm}
Fix a H\"{o}lder measure $\mu$ with support contained in the compact subset $\Sigma$ of $\R$. We define the \emph{weight} $w_{\mu}: \C \to \R^{> 0} \cup \{\infty\}$ associated to $\mu$ by $w_{\mu}(z) = \exp(U^{\mu}(z))$. Given a complex polynomial $P$ and a nonnegative integer $n$, we then define the \emph{$n$-norm of $P$ with respect to $(\mu, \Sigma)$} by
\[\max_{x \in \Sigma} \left(w_{\mu}(x)^n|P(x)|\right) \]
We will write this supremum norm as $\normf{P}$.
\end{defn}

The importance of the $n$-norm comes from the following lemma in weighted potential theory.
\begin{lem}
\label{lem:RSphere}
Choose a compact subset $\Sigma$ of $\R$, and choose a H\"{o}lder probability measure $\mu$ on $\Sigma$. Choose an infinite  sequence $P_1, P_2, \dots $ of real monic polynomials of increasing degree. We assume these polynomials have all roots in $\Sigma$. Suppose
\[\lim_{n \to \infty} \normff{P_n}{\deg P_n}^{1/\deg P_n}  = 1.\]
Then the measures $\mu_{P_1}, \mu_{P_2}, \dots$ converge to $\mu$ in the \weaks topology on $\Sigma$.
\end{lem}
% Tietze extension theorem
\begin{proof}
From Lemma \ref{lem:pot_Hold}, the functions $U^{\mu}$ and $w_{\mu}$ are finite and continuous on $\R$, so the energy $I(\mu)$ is finite. The latter implies that $\Sigma$ has positive capacity. As a result, in the language of \cite[Definition I.1.1]{SaTo97}, the function $w_{\mu}$ is an admissible weight. 

From \cite[Theorem III.1.9]{SaTo97}, there is an infinite sequence of real monic polynomials $Q_1, Q_2, \dots$ with roots contained in $\Sigma$ such that $Q_n$ has degree $n$ for all $n$ and such that
\[\lim_{n \to \infty} \normf{Q_n}^{1/n} = 1.\]
By inserting these polynomials at all missing degrees, we may assume without loss of generality that $P_n$ has degree $n$ for all $n$.

By \cite[Theorem I.3.1]{SaTo97}, the extremal measure associated to the weight $w_{\mu}$ is $\mu$, and the modified Robin constant is $0$ (see \cite[Theorem I.1.3]{SaTo97} for the definition of these terms). The result then follows from \cite[Theorem III.4.2]{SaTo97}.
\end{proof}

The following consequence of the Remez inequality will be used several times.
\begin{lem}
\label{lem:Remez}
Choose a compact finite union of intervals $\Sigma$ and take $\mu$ to be a H\"{o}lder probability measure with support contained in $\Sigma$. Then there is a $C > 0$ depending on $\Sigma$ and $\mu$ so the following holds:

Choose integers $n \ge 0$ and $m \ge 1$, and choose a complex polynomial $P$ of degree $m$. Take $\alpha$ to be some root of $P$, and take $Q(z)$ to be the polynomial $P(z)/(z - \alpha)$. Then
\begin{equation}
\label{eq:Remez}
\normf{Q} \le (2 \cdot\max(n, m))^C \cdot \normf{P}
\end{equation}
\end{lem}
\begin{proof}
Take $\delta$ to be the minimal length of a component in $\Sigma$, and choose $\eta \in (0, 1)$ so $\mu$ is $\eta$-H\"{o}lder. Take
\[\delta_0 = \min\left(\delta, (n+1)^{-1/\eta}\right).\]
Given an interval $I$ of length $\delta_0$ contained in $\Sigma$, there is $C_0 > 0$ depending only on $\mu$, $\Sigma$, and $\eta$ such that
\begin{equation}
\label{eq:Remez_help}
\left|nU^{\mu}(x) - nU^{\mu}(y)\right| \le C_0
\end{equation}
for all $x$ and $y$ in $I$ by \eqref{eq:pot_Hold}.

Take $\Sigma_0$ to be the intersection of $\Sigma$ with the disk of radius $\tfrac{1}{4}m^{-2} \delta_0$ centered at $\alpha$. Given $x_0$ in $\Sigma_0$, there is an interval $I$ contained in $\Sigma$ of length $\delta_0$ containing $x_0$ since $\delta_0 \le \delta$. The ratio of the length of $I \cap \Sigma_0$ to $I$ is at most $m^{-2}/2$. Take $\varphi: \C \to \C$ to be an affine transfomation mapping  $I$ bijectively onto $[-1, 1]$. Then $\varphi(I \cap \Sigma_0)$ has length at most $m^{-2}$. Applying the Remez-type inequality \cite[Theorem 1]{Erde92} to $\frac{1}{A}Q \circ \varphi^{-1}$ with $A = {\max}_{\,x \in I \backslash \Sigma_0} \,|Q(x)|$ gives
\[|Q(x_0)|  \le e^5 \cdot {\max}_{\,x \in I \backslash \Sigma_0} \,|Q(x)|.\]

Combining this with \eqref{eq:Remez_help} gives
\[w_{\mu}(x_0)^n|Q(x_0)| \,\le\, e^{5 + C_0} \cdot {\max}_{\,x \in I \backslash \Sigma_0} \left(w_{\mu}(x)^n|Q(x)|\right).\]
Applying this for all $x_0 \in \Sigma_0$ gives
\begin{align*}
\normf{Q} &\le e^{5 + C_0}\left|\left| w^n_{\mu} Q\right|\right|_{\Sigma \backslash \Sigma_0}\\
& \le\, e^{5 + C_0} \cdot 4m^2 \delta_0^{-1}\cdot \left|\left| w^n_{\mu} P\right|\right|_{\Sigma \backslash \Sigma_0}\\
& \le\, e^{5 + C_0} \cdot 4m^2 \delta_0^{-1} \cdot \normf{P},
\end{align*}
with the second inequality following from the definition of $\Sigma_0$. The lemma follows.
\end{proof}

Our next lemma concerns squarefree linear combinations of integer polynomials. The following definition will be useful.
\begin{defn}
Given a positive integer $k$ and nonzero integer polynomials $Q_1, \dots, Q_k$, we define  $\text{gcd}(Q_1, \dots, Q_k)$
to be the unique integer polynomial $G$ of maximal degree so that
\begin{itemize}
\item The leading term of $G$ is positive,
\item The greatest common denominator of the coefficients of $G$ is $1$ (that is, $G$ is primitive), and
\item For $1 \le i \le k$, $Q_i$ is divisible by $G$ in the polynomial ring over $\QQ$.
\end{itemize}
By Gauss's lemma, we see that $Q_i/G$ is an integer polynomial for $1 \le i \le k$.
\end{defn}

\begin{lem}
\label{lem:squarefree_nocommon}
Choose  integers $k\ge 1$ and $n\ge 0$, and choose nonzero integer polynomials $Q_1, \dots, Q_k$ of degree at most $n$. We assume some $Q_i$ has degree exactly $n$. Take $G = \textup{gcd}(Q_1, \dots, Q_k)$.

There are then nonnegative integers $b_2, \dots, b_k$ no larger than $2(n - \deg G)$ for which the integer polynomial
\[Q = \left(Q_1 + b_2 Q_2 + \dots + b_k Q_k\right)/G\]
is squarefree and has degree $n - \deg G$.
\end{lem}
\begin{proof}
If $n = \deg G$, the assignment $b_2 = \dots =b_k = 0$ satisfies the conditions of the lemma. So assume $n > \deg G$. Take $F$ to be the polynomial in $\C[x_2, \dots, x_k, t]$ defined by
\[F(x_2, \dots, x_k)(t) = (Q_1/G)(t) + x_2\cdot(Q_2/G)(t) + \dots  + x_k\cdot(Q_k/G)(t).\]
Given  polynomials $F_1, F_2 \in \C[x_2, \dots, x_k, t]$ such that $F = F_1F_2$, a consideration of degrees shows that one of $F_1$, $F_2$ lies in $\C[t]$. By the definition of $G$, this implies that one of $F_1, F_2$ lies in $\C$. So $F$ is irreducible. 

Since some $Q_i$ has degree $n$, $F$ has degree $n - \deg G > 0$ as a polynomial over $\C[x_2, \dots, x_k]$. So $\frac{\partial}{\partial t} F$ is nonzero and of smaller degree than $F$ and hence has no common nonconstant factor with $F$. As a consequence, the discriminant $\disc(F)$ of $F$ with respect to $t$ is a nonzero element of $\C[x_2, \dots, x_k]$. The determinantal definition of the discriminant shows this polynomial has total degree at most $2(n - \deg G) - 1$.

Given integers $b_2, \dots b_k$, the determinantal definition of the discriminant also gives us
\[\disc(F) (b_2, \dots, b_k) = \begin{cases} \disc (F(b_2, \dots, b_k)) \quad&\text{ if } F(b_2, \dots, b_k) \text{ has degree } n - \deg G \\ 0 & \text{ otherwise.} \end{cases}\]
An application of the multivariate Gregory--Newton formula \cite[(2) and (3)]{Salz45} shows that there are nonnegative integers $b_2, \dots, b_k$ such that $b_2 + \dots + b_k$ is at most $2(n -\deg G) - 1$ and such that $\disc(F)(b_2, \dots, b_k)$ is nonzero. From the above formula, these integers satisfy the conditions of the lemma.
\end{proof}

\subsection{Approximating polynomials}

Our eventual goal is to find monic integer polynomials whose associated counting measures converge to a given $\mu$. As a first step, it is convenient to have a sequence of monic real polynomials with this behavior.

\begin{defn}
\label{defn:approx}
Choose a compact subset $\Sigma$ of $\R$ and a H\"{o}lder probability measure $\mu$ with support contained in $\Sigma$. For every nonnegative integer $ i \le n$, we take $\alpha_i$ to be the minimal $\alpha$ in $\Sigma$ for which
\[\mu\big((-\infty, \alpha]\big) = i/n.\]

We then define the \emph{approximating polynomial} to $\mu$  of degree $n$ by
\[P_{n, \mu}(z) = \prod_{i = 1}^n  (z - \alpha_i).\]
The definition of $P_{n, \mu}$ gives
\[-n^{-1}\, \le \left(\mu_{P_{n, \mu}} - \mu\right)\big((-\infty, x]\big) \,\le\, 0 \quad\text{for all } x \in \R,\]
so we have
\begin{equation}
\label{eq:notelescope}
\left|\left(\mu_{P_{n, \mu}} - \mu\right)\big(I\big)\right|\, \le\, n^{-1} \quad\text{for any closed interval } I \subseteq \R.
\end{equation}
\end{defn}

\begin{prop}
\label{prop:three_pages}
Given a compact subset $\Sigma$ of $\R$ and a H\"{o}lder probability measure $\mu$ with support contained in $\Sigma$, there is a $C > 0$ determined from $\Sigma$ and $\mu$ so the following holds:

Choose any integer $n \ge 2$, and define $\alpha_1, \dots, \alpha_n$ and $P_{n, \mu}$ from $\mu$ as in Definition \ref{defn:approx}.  Then the energy $I(\mu)$ of $\mu$ satisfies
\begin{equation}
\label{eq:energy_approx}
n^{-C n } \cdot \exp\left(-\tfrac{1}{2} n^2I(\mu)\right) \,\le\, \prod_{1 \le i < j \le n} |\alpha_i - \alpha_j|.
\end{equation}
Furthermore, taking 
\[\rho_{n,\mu}(x)\, =\,\min_{\,1 \le i \le n} |x - \alpha_i|,\]
the weight $w_{\mu}(x) = e^{U^{\mu}(x)}$ satisfies
\begin{equation}
\label{eq:potential_approx}
n^{-C}\cdot \rho_{n, \mu}(x)    \,\le\, w_{\mu}(x)^n \cdot\left|P_{n, \mu}(x)\right|  \,\le\, n^C\cdot \rho_{n, \mu}(x)\quad\text{for all  }\, x\in \Sigma.
\end{equation}
\end{prop}
\begin{proof}
Choose $\eta \in (0, 1)$ so $\mu$ is $\eta$-H\"{o}lder. Choose $C_0 > 1$ so that $\Sigma$ is contained in $[-C_0, C_0]$ and so that \eqref{eq:Holder_condition} holds for $\eta$ with $C = C_0$. From this condition, we find that
\begin{equation}
\label{eq:Holder_spaced}
\alpha_{i+1} - \alpha_i \ge (C_0n)^{-1/\eta}
\end{equation}
for $i \ge 1$.

We take $\mu_n$ to be the counting measure associated to $P_{n, \mu}$ as in Definition \ref{defn:counting}. Given $x \in \R$ and $r > 0$, take $I(x, r)$ to be the interval $[x - r, x+r]$. Then \eqref{eq:Tonelli} gives
\begin{equation}
\label{eq:relTone}
U^{\mu_n}(x) - U^{\mu}(x) =  \int_0^{2C_0} r^{-1} \cdot \left(\mu_n - \mu\right)\big(I(x, r)\big) dr\quad\text{for all } \,x \in [-C_0, C_0].
\end{equation}
Our next goal is to estimate this integral, and we do this by breaking it into pieces. To start, take 
\[\delta = \tfrac{1}{3}(C_0n)^{-1/\eta}.\]
Then $I(x, \delta)$ contains at most one root of $P_{n, \mu}$ for any $x \in \R$ by \eqref{eq:Holder_spaced}. For a given real $x$, suppose $I(x, \delta)$ contains a root $\alpha$ of $P_{n, \mu}$. We then have
\[\int_0^{\delta} r^{-1} \mu_n\big(I(x, r)\big)dr = \int_{|\alpha - x|}^{\delta}n^{-1} r^{-1}dr = - n^{-1} \log\left(\delta^{-1} |\alpha - x|\right) \,\,\text{ in } \,\R \,\cup\, \{\infty\}.\]
If $I(x, \delta)$ contains no root, then this integral is $0$. By considering these cases separately, we see that
\begin{equation}
\label{eq:delta_mun}
\left| n^{-1} \log \rho_{n, \mu}(x) + \int_0^{\delta} r^{-1} \cdot \mu_n\big(I(x, r)\big)dr\right| \le  n^{-1}\log( \max( \delta^{-1}, \,2C_0)),
\end{equation}
for any $x$ in $[-C_0, C_0]$ other than the roots of $P_{n, \mu}$. This is $ \ll n^{-1} \log n$, where the implicit constants here and below will depend just on $\Sigma$, $\mu$, and $C_0$.

Next, for any $x \in \R$, \eqref{eq:Holder_condition} gives
\begin{equation}
\label{eq:delta_mu}
\int_0^\delta r^{-1} \mu(I(x, r))dr \le \int_0^{\delta} r^{-1} C_0(2r)^{\eta}dr = C_0 \eta^{-1}(2\delta)^{\eta}\ll n^{-1} \log n.
\end{equation}
Finally, for any $x \in \R$, we have
\begin{align}
\label{eq:del3c0}
\left|\int_{\delta}^{2C_0} r^{-1}(\mu_n - \mu)(I(x, r))dr\right| &\,\le\, \int_{\delta}^{2C_0}r^{-1} n^{-1} dr  \ll n^{-1} \log n
\end{align}
by \eqref{eq:notelescope}.

Summing \eqref{eq:delta_mun}, \eqref{eq:delta_mu}, and \eqref{eq:del3c0} then gives
\[\left| n^{-1} \log \rho_{n, \mu}(x) + \int_0^{2C_0} r^{-1} \cdot (\mu_n -\mu)\big(I(x, r)\big)dr\right|\, \ll\,  n^{-1} \log n.\]
By \eqref{eq:relTone}, we thus have
\begin{equation}
\label{eq:explicit_pot_approx}
\left|nU^{\mu}(x) - nU^{\mu_n}(x) - \log \rho_{n, \mu}(x)\right| \ll  \log n
\end{equation}
for any $x$ in $[-C_0, C_0]$ that is not a root of $P_{n, \mu}$. Since $P_{n, \mu}(x) = \exp(-nU^{\mu_n}(x))$, we may exponentiate this inequality to find that \eqref{eq:potential_approx} holds for sufficiently large $C$. As \eqref{eq:potential_approx} is vacuously true at the roots of $P_{n, \mu}$, we find that \eqref{eq:potential_approx} holds with this $C$ for all $x$ in $[-C_0, C_0]$.

We turn to the estimate for $I(\mu)$. Given $k \le n$, applying \eqref{eq:Tonelli} and \eqref{eq:Holder_condition} gives
\[\int_{\alpha_k - \delta}^{\alpha_k + \delta} \log |x - \alpha_k| d\mu(x) \ge\log \delta \cdot C_0(2\delta)^{\eta} - \int_0^\delta r^{-1} \mu(I(\alpha_k, r))dr.\]
By \eqref{eq:delta_mu}, this is $\gg  -n^{-1} \log n$. Since $\log|\rho_{n, \mu}(x)|$ is at least $\log \delta$ for $x$ outside the union of intervals $\bigcup_{k \le n} [\alpha_k - \delta, \alpha_k + \delta]$, we thus have
\begin{equation}
\label{eq:deliver_me}
\int \log |\rho_{n, \mu}(x)| d\mu(x) \gg - \log n. 
\end{equation}

Given $k \le n$, we have
\[\sum_{\substack{1 \le i \le n \\ i \ne k}}\log |\alpha_i - \alpha_k| = \lim_{x \to \alpha_k} -nU^{\mu_n}(x) - \log \rho_{n, \mu}(x),\]
so \eqref{eq:explicit_pot_approx} gives
\[\sum_{\substack{1 \le i, k \le n \\ i \ne k}} \log |\alpha_i - \alpha_k| + \sum_{i = 1}^n nU^{\mu}(\alpha_i)   \gg -n\log n.\]
The second sum  may be rewritten using the identity
\[\sum_{i=1}^n nU^{\mu}(\alpha_i) = -n^2\int\log|x - t| d\mu(x) d\mu_n(t) = n^2\int U^{\mu_n} d\mu,\]
and \eqref{eq:explicit_pot_approx} and \eqref{eq:deliver_me} give
\[n^2 \int_{\Sigma} (U^{\mu_n} - U^{\mu})d\mu  \ll n \log n.\]
So
\[\sum_{\substack{1 \le i, k \le n \\ i \ne k}} \log |\alpha_i - \alpha_k|  + n^2I(\mu) \gg  -n \log n,\]
and \eqref{eq:energy_approx} follows for sufficiently large $C$.
\end{proof}

\section{Squarefree  polynomials and the geometry of numbers}
\label{sec:GON}
The main intermediate result on our way to Theorem \ref{thm:main} is the following theorem on squarefree integer polynomials.
\begin{thm}
\label{thm:squarefree}
Choose a compact subset $\Sigma$ of $\R$, and choose a H\"{o}lder probability measure $\mu$ on $\Sigma$. Suppose 
\[\int \log|Q| d\mu \ge 0\]
for every nonzero integer polynomial $Q$.

Then there is a $C > 0$ determined from $\mu$ and $\Sigma$ so, for all degrees $n \ge 2$, there is a squarefree integer polynomial $R_n$ of degree $n$  satisfying
\[\normf{R_n} \le n^{C \sqrt{n}}.\]
\end{thm}

The proof of this requires two different applications of the geometry of numbers. First, we have the following application of the flatness theorem to integer polynomials. This theorem will also be needed in the path from Theorem \ref{thm:squarefree} to Theorem \ref{thm:main}.
\begin{thm}
\label{thm:adjust}
There is a positive real number $C$ so we have the following:

Take $c$ and $n$ to be positive integers, and suppose we have a squarefree integer polynomial $Q$ with factorization 
\[Q(z) = c (z - \alpha_1) (z - \alpha_2) \dots (z - \alpha_n)\]
over $\C$, so the $\alpha_i$ are all distinct. For $i \le n$, define a degree $n-1$ polynomial
\[Q_i(z) = Q(z)/(z - \alpha_i).\]

Given any real polynomial $P(z)$ of degree at most $n - 1$, there then are complex numbers $\beta_1, \dots, \beta_n$ satisfying
\[\sum_{i \le n} |\beta_i| \le C n \log 2n\]
for which
\[P(z) - \sum_{i \le n} \beta_i Q_i(z)\]
 is an integer polynomial.
\end{thm}

\begin{proof}
Take $E = \QQ(\alpha_1, \dots, \alpha_n)$. For $i$ in $\{1, 2, \dots, n\}$ and $j$ in $\{0, 1, \dots, n-1\}$, take $b_{ij}$ to be the degree $j$ coefficient of $Q_i$. These coefficients are all algebraic integers. Otherwise, some $b_{ij}$ would have negative valuation at some prime of $E$, and this cannot happen by Gauss's lemma since discrete valuation rings are unique factorization domains.

Note that a field automorphism of $E$ that takes $\alpha_i$ to $\alpha_k$ will also take $Q_{i}$ to $Q_{k}$. For any rational integers $d_0, \dots, d_{n-1}$, we may conclude that
\[\left\{ \sum_{0 \le j \le n-1} d_j b_{ij} \,:\,\, 1 \le i \le n\right\}\]
is the set of roots to a monic integer polynomial. In particular, if any element in this set is nonzero, there is some element in this set of absolute value at least one. But the $Q_i$ are linearly independent, so these sums are all zero only if $d_j = 0$ for each $j$.

Take $Q^{\circ}_1, \dots, Q^{\circ}_n$ to be a list of polynomials containing $Q_i$ for every $i \le n$ for which $\alpha_i$ is real and containing 
\[\tfrac{1}{\sqrt{2}}\left(Q_i + \overline{Q_i}\right)\quad\text{and}\quad \tfrac{i}{\sqrt{2}} (Q_i - \overline{Q_i})\]
for every $i \le n$ for which $\overline{\alpha_i} = \alpha_j$ for some $j > i$. For every sequence of rational integers $d_0, \dots, d_{n-1}$ that are not all zero, the above argument shows there is some $i \le n$ so that, if we write $Q^{\circ}_i(z)$ in the form $b^{\circ}_{n-1}z^{n-1} + \dots + b^{\circ}_0$, we have
\[\left| d_0b^{\circ}_0 + \dots + d_{n-1}b^{\circ}_{n-1}\right| \ge 1.\]
 We can conclude from the flatness theorem for simplices \cite[Corollary 2.5]{BLPS_flat99} that, for some absolute $C_0 > 0$, any translate of the convex body
\[K = \left\{ \sum_{i \le n} \beta_i Q^{\circ}_i \,:\,\,  \sum_{i \le n} |\beta_i| \le C_0 n \log 2n \right\}\]
contains an integer polynomial. This suffices to prove the theorem.
\end{proof}

Combining this theorem with the Remez inequality gives the following useful corollary.
\begin{cor}
\label{cor:adjust}
Fix a compact finite union of intervals $\Sigma$ and a H\"{o}lder probability measure $\mu$ with support contained in $\Sigma$. Choose an integer $n > 1$, a real polynomial $P$ of degree at most $n-1$, and a squarefree integer polynomial $Q$ of degree $n$.

Then there is an integer polynomial $R$ of degree at most $n-1$ so that
\[\normf{(P - R)} \le n^C \cdot \normf{Q},\]
where $C > 0$ just depends on $\Sigma$ and $\mu$.
\end{cor}
\begin{proof}
Take $R$ to be the polynomial $P - \sum_{i \le n}\beta_i Q_i$ constructed in Theorem \ref{thm:adjust}. The upper bound on the norm of $P - R$ follows from Lemma \ref{lem:Remez} and the bounds on the $\beta_i$.
\end{proof}

Our second use of the geometry of numbers is an application of Minkowski's results on successive minima to integer polynomials. The precedent for this result comes from Hilbert \cite{Hilbert94}, who proved a similar result in the case where $\Sigma$ is an interval and $\mu$ is the unweighted equilibrium distribution, and from Amoroso \cite{Amor90} and Pritsker \cite[Theorem 2.1]{Prit05}, who proved weighted versions of Hilbert's result.
\begin{prop}
\label{prop:Mink}
Choose an integer $m \ge 0$, distinct real numbers $\alpha_1, \dots, \alpha_{m+1}$, and positive real numbers $v_1, \dots, v_{m+1}$. Define $K$ to be the set of real polynomials $P$ of degree at most $m$ that satisfy
\[v_i \cdot |P(\alpha_i)| \le 1 \quad\text{for all }\, 1 \le i \le m+1.\]
For a nonnegative integer $i$ satisfying $i \le m$, take $\lambda_i$ to be the least real number so $\lambda_i K$ contains at least $i+1$ linearly independent integer polynomials. Then
\[\frac{1}{(m+1)!}D  \,\le\,  \lambda_0 \cdot\lambda_1\cdot \dots \cdot\lambda_m \,\le\,  D ,\]
where we have taken the notation
\begin{equation}
\label{eq:Mink_D}
D = \prod_{i = 1}^{m+1} v_i\,\cdot\, \prod_{1 \le i < j \le m+1} |\alpha_i - \alpha_j|.
\end{equation}
\end{prop}
\begin{proof}
By identifying the real polynomial $a_0 + a_1z + \dots + a_mz^m \in \R[z]$ with the tuple $(a_0, a_1, \dots, a_m)$, we may think of $K$ as a convex, centrally-symmetric subset of $\R^{m+1}$, with $(a_0, a_1, \dots, a_m)$ lying in $K$ if and only if
\begin{equation}
\label{eq:K_alt}
\big|a_0 +a_1\alpha_k \dots + a_m\alpha_k^{m}\big| \le v_k^{-1}  \quad\text{for } 1 \le k \le m+1.
\end{equation}
Minkowski's second theorem \cite[p. 376]{Gruber07} then shows
\[ \frac{2^{m+1}}{(m+1)!}   \Vol(K)^{-1} \,\le\,\lambda_0 \cdot\lambda_1\cdot \dots \cdot\lambda_m\,\le\,  2^{m+1} \Vol(K)^{-1}.\]
To prove the proposition, we now just need to estimate the volume of $K$. Consider the Vandermonde matrix
\[V = \left(\begin{matrix} 1 & \alpha_1 & \dots & \alpha_1^m \\ 1 & \alpha_2 & \dots & \alpha_{2}^m \\ \vdots & \vdots & \ddots & \vdots \\ 1 & \alpha_{m+1} & \dots & \alpha_{m+1}^m \end{matrix} \right).\]
Take
\[K_1 := \left[-v_1^{-1},\, v_1^{-1}\right] \times \dots \times \left[-v_{m+1}^{-1},\, v_{m+1}^{-1}\right] \subseteq \R^{m+1}.\]
Then \eqref{eq:K_alt} is equivalent to the statement that $K$ is the preimage of $K_1$ under the linear map $V$. The set $K_1$ has volume $2^{m+1} \cdot \prod_{i=1}^{m+1} v_i^{-1}$, and the Vandermonde matrix has determinant
\[\det V = \prod_{1 \le i < j \le m+1}\left(\alpha_j - \alpha_i\right),\]
so we may conclude that 
\[\text{Vol}\, K = |\det V|^{-1} \cdot \text{Vol} \, K_1 = 2^{m+1} \cdot \prod_{i=1}^{m+1} v_i^{-1} \cdot  \prod_{1 \le i < j \le m+1}\left|\alpha_i - \alpha_j\right|^{-1}. \]
The result follows.
\end{proof}

\begin{prop}
\label{prop:squeeze}
Take $\mu$ to be a H\"{o}lder probability measure with support contained in a compact subset $\Sigma$ of $\R$. Choose nonnegative integers $n$ and $m$ and a nonzero integer polynomial $R$, and take $Q_0, Q_1, \dots, Q_m$ to be linearly independent integer polynomials of degree at most $m$. Then, if $\int_{\Sigma} \log|R| d\mu$ is nonnegative, we have
\[\prod_{i = 0}^m \normf{Q_i R} \ge\frac{1}{(m+1)!} \cdot \exp\left(\left(mn - \tfrac{1}{2}m^2  + n - \tfrac{1}{2}m\right) \cdot  I(\mu)\right).\]
\end{prop}
\begin{proof}
Given $x_1, \dots, x_{m+1} \in \Sigma$, define
\[F(x_1, \dots, x_{m+1}) = \prod_{i = 1}^{m+1} |R(x_i)| w_{\mu}^n(x_i) \cdot \prod_{1 \le i  < j \le m+1} |x_i - x_j|.\]
We then have
\begin{align*}
&\int_{\Sigma} \int_{\Sigma} \dots \int_{\Sigma} \log\left|F(x_1, \dots, x_{m+1})\right| d\mu(x_1) d\mu(x_2) \dots d\mu(x_{m+1})\\
&\qquad = \, (m+1) \cdot \left(\int_{\Sigma} (\log|R| + nU^{\mu}) d\mu \right)\,-\, \tfrac{1}{2}(m^2 + m)\cdot I(\mu)\\
&\qquad  \ge\, \left(mn - \tfrac{1}{2}m^2  + n - \tfrac{1}{2}m\right)\cdot  I(\mu),
\end{align*}
with the equality following from the linearity of integration and the inequality following from the assumption that $\int_{\Sigma} \log|R| d\mu$ is nonnegative. From this inequality, we may choose $\alpha_1, \dots, \alpha_{m+1}$ in $\Sigma$ for which we have
\begin{equation}
\label{eq:Falpha}
F(\alpha_1, \dots, \alpha_{m+1}) \ge \exp\left(\left(mn - \tfrac{1}{2}m^2  + n - \tfrac{1}{2}m\right)\cdot  I(\mu)\right).
\end{equation}

Permuting the $Q_k$ if necessary, we assume that
\[\max_{1 \le i \le m+1} \left(w_{\mu}^n(\alpha_i) \cdot \big|(R\cdot Q_k)(\alpha_i) \big|\right)\]
is nondecreasing as $k$ increases. If we take $v_i$ to equal to $|R(\alpha_i)|w_{\mu}^n(\alpha_i)$ for $1 \le i \le m+1$, we see that the right hand side of \eqref{eq:Mink_D} equals $F(\alpha_1, \dots, \alpha_{m+1})$.  Applying Proposition \ref{prop:Mink}, we see that the $\lambda_k$ appearing in the proposition statement satisfy
\[\lambda_k  \,\le\, \max_{1 \le i \le m+1} \left(w_{\mu}^n(\alpha_i) \cdot \big|(R\cdot Q_k)(\alpha_i) \big| \right) \,\le\, \normf{RQ_k}\]
for $k \le m$, so the proposition gives
\[\frac{F(\alpha_1, \dots, \alpha_{m+1})}{(m+1)!}   \,\le\, \prod_{k=0}^m \normf{RQ_k}, \]
and the result follows from \eqref{eq:Falpha}.
\end{proof}

In the case where $R = 1$, $n = m$, and $\Sigma$ is a compact finite union of intervals, this bound is nearly sharp. Specifically, we have the following result.
\begin{prop}
\label{prop:Hilbert}
Choose a compact finite union of intervals $\Sigma$, choose a  H\"{o}lder probability measure $\mu$ with support contained in $\Sigma$, and choose an integer $n \ge 2$. Then there is a sequence $Q_0, \dots, Q_n$ of linearly independent integer polynomials of degree at most $n$ so
\begin{equation}
\label{eq:Hilbert}
\prod_{i = 0}^n \normf{Q_i} \,\le\, n^{Cn} \exp\left(\tfrac{1}{2} n^2 I(\mu)\right),
\end{equation}
where $C> 0$ depends just on $\mu$ and $\Sigma$.
\end{prop}
\begin{proof}
For a given integer $n \ge 2$, take $P_{n + 1, \mu}$ to be the degree $n+1$ approximating polynomial to $\mu$ constructed in Definition \ref{defn:approx} and take $\alpha_1, \dots, \alpha_{n+1}$ to be the roots of $P_{n+1, \mu}$. For $k$ satisfying $1 \le k \le n+1$, take $P_k$ to be the polynomial $P_{n+1, \mu}(z)/(z - \alpha_k)$.

Proposition \ref{prop:three_pages} gives
\[\log \normff{P_{n+1, \mu}}{n+1} \ll \log n,\]
where the implicit constants here and throughout the proof just depend on $\Sigma$ and $\mu$. Lemma \ref{lem:Remez} gives
\[\log \normff{P_k}{n+1} - \log \normff{P_{n+1, \mu}}{n+1} \ll \log n,\]
so $\log \normff{P_k}{n+1} \ll \log n$ for all $k \le n$. We have $|\log w_{\mu}(x)| \ll 1$ for all $x \in \Sigma$ by Lemma  \ref{lem:pot_Hold}, so
\begin{equation}
\label{eq:Pk_normf}
\log \normf{P_k} \ll \log n.
\end{equation}

Given real numbers $b_1, \dots, b_{n+1}$ that are not all $0$, if we take $P  = \sum_{k = 1}^{n+1} b_kP_k$, we find
\[\normf{P} \le \sum_{k = 1}^{n+1} |b_k| \normf{P_k}.\]
Taking logarithms and applying \eqref{eq:Pk_normf} gives
\[\log \normf{P} - \max_{k \le n+1} \log |b_k| \ll \log n.\]
But we have $P(\alpha_k) = b_kP_k(\alpha_k)$ for $k \le n+1$, so this gives
\begin{equation}
\label{eq:normf_interp}
\log \normf{P} - \max_{k\le n+1}  \log \left|\frac{P(\alpha_k)}{P_k(\alpha_k)}\right| \ll \log n.
\end{equation}
Since the polynomials $P_1, \dots, P_{n+1}$ are linearly independent, this inequality holds for all nonzero real polynomials $P$ of degree at most $n$.

Applying Proposition \ref{prop:Mink} to $\alpha_1, \dots, \alpha_{n+1}$ with $v_k$ equal to $|P_k(\alpha_k)|^{-1}$ for $k \le n+1$ gives that there are linearly independent integer polynomials $Q_0, \dots, Q_n$ satisfying
\[\prod_{i=0}^n\left( \max_{k \le n+1} \left|\frac{Q_i(\alpha_k)}{P_k(\alpha_k)}\right| \right) \,\le\, \prod_{k=1}^{n+1} |P_k(\alpha_k)|^{-1} \cdot \prod_{1 \le i < j \le n+1} |\alpha_i - \alpha_j|.\]
The right hand side of this is $\prod_{1 \le i < j \le n+1} |\alpha_i - \alpha_j|^{-1}$, which satisfies
\[\log \bigg(\prod_{1 \le i < j \le n+1} |\alpha_i - \alpha_j|^{-1}\bigg) - \tfrac{1}{2} n^2 I(\mu) \ll n \log n\]
by \eqref{eq:energy_approx}. Combining these inequalities with \eqref{eq:normf_interp} gives
\[\sum_{i = 0}^n\log \normf{Q_i} - \tfrac{1}{2}n^2 I(\mu) \ll n \log n.\]
Exponentiating this establishes the proposition.
\end{proof}

In the context of Theorem \ref{thm:main}, the following consequence of Proposition \ref{prop:Hilbert} accounts for the fact that the discriminant of a squarefree integer polynomial is a nonzero integer. The assumption that $\mu$ is H\"{o}lder can be removed, as we will show in Proposition \ref{prop:general_energy}.
\begin{cor}
\label{cor:neg_energy}
Suppose $\mu$ is a H\"{o}lder probability measure on a compact subset $\Sigma$ of $\R$. Then, if $\int_{\Sigma} \log|Q| d\mu$ is nonnegative for every nonzero integer polynomial $Q$, the energy of $\mu$ satisfies
\[I(\mu) \le 0.\]
\end{cor}
\begin{proof}
For $n \ge 2$, Proposition \ref{prop:Hilbert} shows there is a nonzero integer polynomial $Q_n$ with $\normf{Q_n} \le n^C  \exp\left(\tfrac{1}{2} n I(\mu)\right)$. It follows that
\[0 \,\le\, \int_{\Sigma} \log|Q_n|  d\mu \,\le\, C \log n + \tfrac{1}{2} n I(\mu) - n\int_{\Sigma} U^{\mu}d\mu \,=\, C\log n - \tfrac{1}{2} n I(\mu).\]
So $C \log n \ge \tfrac{1}{2} n I(\mu)$ for all $n \ge 2$, and this implies that $I(\mu)$ is nonpositive.
\end{proof}

The final lemma we need combines Theorem \ref{thm:adjust} and Proposition \ref{prop:squeeze}.
\begin{lem}
\label{lem:second_squeeze}
Suppose $\mu$ is a H\"{o}lder probability measure supported on a compact finite union of intervals $\Sigma$. There is then $C > 0$ so we have the following:

Choose nonnegative integers $k, m, n$ with $k \le m \le n$ and $n \ge 2$. Choose linearly independent integer polynomials $Q_0, \dots, Q_m$ of degree at most $n$ satisfying
\begin{equation}
\label{eq:ordah}
\normf{Q_0} \le \normf{Q_1} \le \dots \le \normf{Q_m}.
\end{equation}
Take $G = \textup{gcd}(Q_0, \dots, Q_k)$. We assume that $\int_{\Sigma} \log|G| d\mu$ is nonnegative and that
\[ m =  \max_{0 \le i \le k} \deg Q_i/G.\]
Then
\begin{equation}
\label{eq:second_squeeze}
\prod_{i = 0}^m \normf{Q_i} \ge n^{-Cn} \exp\left(\left(mn - \tfrac{1}{2}m^2 \right)\cdot I(\mu)\right).
\end{equation}
\end{lem}
\begin{proof}
Choose $C_0 > 0$ so Lemma \ref{lem:Remez} and Theorem \ref{thm:adjust} hold for $(\mu, \Sigma)$ with $C = C_0$, and so $\Sigma$ lies in $[-C_0, C_0]$.

If $k$ is $0$, then $m =k$ and the Lemma is immediate from Proposition \ref{prop:squeeze}. So we may assume $k \ge 1$. 

Applying Lemma \ref{lem:squarefree_nocommon}, we choose integers $b_{0}, \dots, b_{k-1}$ in $[0, 2m]$  for which the integer polynomial
\[Q = \left(Q_k + b_{k-1}Q_{k-1} + \dots + b_{0} Q_0\right)/G.\]
is squarefree and has degree $m$. From \eqref{eq:ordah}, we then have
\begin{equation}
\label{eq:RGk}
\normf{QG} \le \left(1 + \sum_{i = 0}^{k-1} b_i\right)  \cdot \normf{Q_k} \le 3km \normf{Q_k}.
\end{equation}

Take $\alpha_{1}, \dots, \alpha_{m}$ to be the roots of $Q$. Choose a nonnegative integer $j < m$. Applying Theorem \ref{thm:adjust} to $Q$ and $P = \tfrac{1}{2}z^j$, we find that there are complex numbers $\beta_1, \dots, \beta_{m}$ satisfying
\begin{equation}
\label{eq:loc_beta}
\sum_{i \le m} |\beta_i| \le C_0 m \log 2m
\end{equation}
for which
\[\tfrac{1}{2}z^j - \sum_{i \le m} \beta_i Q(z)/(z - \alpha_i)\]
is an integer polynomial. We then define an integer polynomial
\[H_j(z) = \sum_{i \le m} 2\beta_i Q(z)/(z - \alpha_i).\]
This polynomial has degree at most $m - 1$. It satisfies
\begin{align}
\label{eq:HjGk}
\normf{H_jG} &\le \sum_{i \le m} 2|\beta_i| (2n)^{C_0} \cdot \normf{QG} \\
& \le  C_0(2n)^{C_0} \cdot 6km^2 \log 2m \cdot \normf{Q_k}, \nonumber
\end{align}
with the first inequality following from Lemma \ref{lem:Remez} and the second following from \eqref{eq:RGk} and \eqref{eq:loc_beta}.

Now unfixing $j$, we see that the polynomials $H_0, \dots, H_{m - 1}$ are linearly independent, as their images in $\FFF_2[z]$ are linearly independent.  In particular, since $Q_i$ has degree $m + \deg G$ for some $i \le k$, we find that there is a subset $S$ of $\{0, \dots, m- 1\}$ of cardinality $m - k$ so the polynomials
\[\{Q_0/G,\, \dots,\, Q_k/G\} \cup \{ H_j\,:\,\, j \in S\}\]
are a basis for the vector space of polynomials of degree at most $m$. Applying Proposition \ref{prop:squeeze} with $R = G$ gives
\[\prod_{i = 0}^k \normf{Q_i} \cdot \prod_{j \in S} \normf{H_jG} \ge \frac{ \exp\left(\left(mn - \tfrac{1}{2}m^2  + n - \tfrac{1}{2}m\right) \cdot  I(\mu)\right)}{(m+1)!}.\]
From \eqref{eq:HjGk} and \eqref{eq:ordah}, we also have
\[\prod_{j \in S} \normf{H_jG} \le  \left(C_0(2n)^{C_0} \cdot 6km^2 \log 2m\right)^{m-k}  \cdot \prod_{i = k+1}^{m}  \normf{Q_i}.\]
Combining the previous two inequalities gives the lemma.
\end{proof}

We can now prove Theorem \ref{thm:squarefree}.
\begin{proof}[Proof of Theorem \ref{thm:squarefree}]
If $\Sigma'$ contains $\Sigma$, we see that the $n$-norm of a polynomial with respect to $(\mu, \Sigma)$ is at most the $n$-norm with respect to $(\mu, \Sigma')$. So it suffices to prove the theorem in the case where $\Sigma$ is a closed interval.

Choose $C_0 > 0$ so Lemma \ref{lem:second_squeeze}, Proposition \ref{prop:Hilbert}, and Lemma \ref{lem:Remez} hold with $C = C_0$, and so $\Sigma$ lies in $[-C_0, C_0]$.

Now fix $n \ge 2$ as in the theorem statement. By Proposition \ref{prop:Hilbert}, there are linearly independent integer polynomials $Q_0, \dots, Q_n$ of degree at most $n$ satisfying
\begin{equation}
\label{eq:thm_Hilbert}
\prod_{i = 0}^n \normf{Q_i} \,\le\, n^{C_0n} \exp\left(\tfrac{1}{2} n^2 I(\mu)\right).
\end{equation}
Permuting if necessary, we assume that 
\begin{equation}
\label{eq:ordah2}
\normf{Q_0} \le \normf{Q_1} \le \dots \le \normf{Q_n}.
\end{equation}

For $0 \le k \le n$, take $G_k$ to be $\textup{gcd}(Q_0, \dots, Q_k)$, and take
\[m_k = - \deg G_k \,+ \max_{0 \le i \le k} \deg(Q_i).\]
From this definition, we always have $m_0 = 0$. In addition, the linear independence of $Q_0, \dots, Q_k$ implies that 
\[m_k \ge k\quad\text{for }\,0 \le k \le n.\]
We may apply Lemma \ref{lem:second_squeeze} to $Q_0, \dots, Q_{m_k}$ with $m = m_k$ for any integer $k$ in $[0, n]$ Dividing \eqref{eq:thm_Hilbert} by \eqref{eq:second_squeeze} then gives
\[\prod_{i = m_k+1}^n \normf{Q_i} \le n^{2C_0 n} \exp\left(\tfrac{1}{2} (n-m_k)^2 \cdot I(\mu)\right) \le n^{2C_0n},\]
with the second inequality following from Corollary \ref{cor:neg_energy}. From \eqref{eq:ordah2}, we thus have
\[\normf{Q_{m_k + 1}}^{n - m_k} \le n^{2C_0 n} \quad\text{for }\, 0 \le k < n\,\text{ if }\, m_k < n.\]
Now choose $k$ to be the maximal nonnegative integer so $n - m_k \ge n^{1/2}$. Since $m_0 = 0$ and $m_j \ge j$ for each $j$, this integer exists and lies in $[0, n)$. We then have
\begin{equation}
\label{eq:Qmk1}
\normf{Q_{m_k+1}} \le n^{2C_0 \sqrt{n}}.
\end{equation}
From Lemma \ref{lem:squarefree_nocommon}, we may find a squarefree integer polynomial $Q$ of degree
\[m_{m_k + 1}> n - n^{1/2}\]
and integers $b_0, \dots, b_{m_k}$  in $[0, 2n]$ so
\[QG_{m_k+1} = Q_{m_k + 1} + b_{m_k}Q_{m_k} + \dots + b_0 Q_0.\]
From \eqref{eq:Qmk1} and \eqref{eq:ordah2}, we then have
\begin{equation}
\label{eq:Rgmk1}
\normf{QG_{m_k+1}}\le 3n^{2 + 2C_0 \sqrt{n}}
\end{equation}
The polynomial $G_{m_k+1}$ has degree smaller than $n^{1/2}$. Iterating Lemma \ref{lem:Remez} then gives
\begin{equation}
\label{eq:raw_Remez}
\normf{Q} \le (2n)^{C_0 \sqrt{n}} \normf{Q G_{m_k+1}}.
\end{equation}
Finally, we may choose a subset $S$ of $\{0, \dots, n-1\}$ of cardinality $n - \deg Q$ so 
\[R_n(z) = Q(z) \cdot \prod_{\alpha \in S} (z - \alpha)\]
is a squarefree integer polynomial of degree $n$. Since $Q$ has degree at least $n - n^{1/2}$, and since $\Sigma$ lies in $[-C_0, C_0]$, we have
\begin{equation}
\label{eq:random_roots}
\normf{R_n} \le (C_0 + n)^{\sqrt{n}} \cdot \normf{Q}.
\end{equation}
Combining \eqref{eq:Rgmk1}, \eqref{eq:raw_Remez}, and \eqref{eq:random_roots} then gives
\[\normf{R_n} \le (C_0 + n)^{\sqrt{n} }\cdot  (2n)^{C_0 \sqrt{n}}  \cdot 3 n^{2 + 2C_0 \sqrt{n}}.\]
This is less than $n^{C \sqrt{n}}$ for sufficiently large $C > 0$ not depending on $n \ge 2$, and the theorem follows.
\end{proof}
\section{Finding real polynomials to adjust}
\label{sec:adjust}
Suppose we have $\Sigma$ and $\mu$  satisfying the conditions of Theorem \ref{thm:squarefree}. This theorem then produces a sequence $R_2, R_3, \dots$ of squarefree integer polynomials of increasing degree. If these polynomials had all their roots contained in $\Sigma$, and if they were known to be monic and irreducible, we could apply Lemma \ref{lem:RSphere} to finish the proof of Theorem \ref{thm:main} for $(\mu, \Sigma)$. But there is no reason to expect the polynomials to satisfy these extra conditions.

Another possibility is to instead start with the approximating polynomials $P_{2, \mu}, P_{3, \mu}, \dots$ constructed in Definition \ref{defn:approx}, which do have all their roots in $\Sigma$ but which are not generally irreducible integer polynomials. Applying Corollary \ref{cor:adjust} with the squarefree polynomial $Q = R_{n}$ lets us adjust $P_{n, \mu}$ to a monic integer polynomial $R$. Using the Eisenstein condition, we may force these adjusted polynomials to be irreducible; see the proof of Theorem \ref{thm:main} below for details.

The only issue with this approach is that the process of adjusting $P_{n, \mu}$ to $R$ may shift roots to lie off of $\Sigma$. A standard method of showing a real polynomial $P$ has a root in a given interval $(x, y)$ is to show that $P(x)$ and $P(y)$ have different signs. But this method cannot be used to show that the polynomial $R$ has all real roots. After all, while we can use \eqref{eq:potential_approx} and \eqref{eq:Holder_spaced} to show there is some $x$ between any two adjacent roots of $P_{n, \mu}$  such that
\[w^n_{\mu}(x)\cdot|P_{n, \mu}(x)| \ge n^{-C},\]
the best bound we have for $P_{n, \mu} - R$ at this $x$ is
\[w^n_{\mu}(x) \cdot \left|P_{n, \mu}(x) - R(x)\right| \le n^{C\sqrt{n}},\]
where $C>0$ depends just on $\Sigma$ and $\mu$. In particular, $R(x)$ has no reason to have the same sign as $P_{n, \mu}(x)$.

As in the above incomplete argument, our actual proof of Theorem \ref{thm:main} uses Theorem \ref{thm:squarefree} and Corollary \ref{cor:adjust} to adjust real polynomials to integer polynomials, and our method for keeping the roots of the resulting polynomials in $\Sigma$ is to look for sign changes. The key trick of our approach is that, rather than start with the polynomials $P_{n, \mu}$, we start with real polynomials whose high degree terms are already integers, as in Proposition \ref{prop:early_ints_better}. This makes the adjustment from the real polynomials to integer polynomials small enough that the adjustment does not push the roots off $\Sigma$.

\begin{prop}
\label{prop:early_ints_better}
Choose a compact finite union of intervals $\Sigma$, a H\"{o}lder probability measure $\mu$ on $\Sigma$, and a positive number $B$. Take $\kappa$ to be the capacity of $\Sigma$, as defined in Definition \ref{defn:capacity}; we assume $\kappa > 1$. Then there is a real number $C > 3$ depending on $\Sigma$, $\mu$, and $B$ so we have the following:

Choose a positive integer $n > C$, and take 
\[m =\left\lfloor B\sqrt{n} \log n\right\rfloor.\]
 Then there is a real monic polynomial $P_n$ of degree $n$ such that
\begin{enumerate}
\item $P_n$ is squarefree and has all roots in $\Sigma$,
\item We have 
\begin{equation}
\label{eq:early_ints_normf}
\normff{P_n}{n-m} \le n^C \kappa^{m},
\end{equation}
\item For any integer $k$ in the interval $[n-m, n-1]$, the degree $k$ coefficient of $P_n$ is an even integer, and 
\item Taking $\alpha_1< \alpha_2 < \dots < \alpha_n$ to be the roots of $P_n$, there are sequences of real numbers $x_1, \dots, x_n$ and $y_1, \dots, y_n$ such that
\[x_1 < \alpha_1 < y_1 < x_2 < \alpha_2 < y_2 < \dots < x_n < \alpha_n < y_n,\]
such that $[x_k, y_k]$ is a subset of $\Sigma$ for all $k \le n$, and such that
\begin{equation}
\label{eq:big_strong_interval}
\min\Big(w^{n-m}_{\mu}(x_k)\left|P_n(x_k) \right|,\,\,\, w^{n-m}_{\mu}(y_k)\left|P_n(y_k) \right| \Big) \ge n^{-C} \kappa^{m}\quad\text{for }\, k \le n.
\end{equation}
\end{enumerate}
\end{prop}
We will prove this proposition in Section \ref{ssec:early_ints}. Before we do that, we will show it implies our main theorem.

% P overload: P real? P real Q int aux R adjustee?

\subsection{Proof of Theorem \ref{thm:main}}
\label{ssec:main}
By Proposition \ref{prop:limit_Holder}, we may assume without loss of generality that $\mu$ is H\"{o}lder and that $\Sigma$ is a compact finite union of intervals of capacity $\kappa > 1$. Choose $C_0 > 0$ such that Theorem \ref{thm:squarefree} and Corollary \ref{cor:adjust} hold for $(\Sigma, \mu)$ with $C = C_0$. Take $B = 2C_0/\log \kappa$, and choose $C_1 > 3$ such that the condition of Proposition \ref{prop:early_ints_better} holds for this $B$ if $C = C_1$.

Choose an integer $n$ satisfying the inequalities
\begin{equation}
\label{eq:main_n_ass}
n > C_1, \quad 4n^{C_0 + C_0\sqrt{n}} < n^{-C_1}\kappa^{\lfloor B \sqrt{n} \log n\rfloor},\quad\text{and}\quad n - B \sqrt{n}\log n \ge 2.
\end{equation}
We note this is satisfied for all sufficiently large $n$. Take $m = \left\lfloor B\sqrt{n} \log n\right\rfloor$.

Take $P_n$ to be the polynomial of degree $n$ constructed in Proposition \ref{prop:early_ints_better} for $\mu$, $n$, $m$, and $B$. We write this polynomial in the form $z^n + a_{n-1}z^{n-1} + \dots + a_0$, and we take
\[P_{\text{low}}(z) = a_{n-m-1}z^{n-m-1} + a_{n-m-2}z^{n-m-2} + \dots + a_0.\]
Then $P_n - P_{\text{low}}$ is an integer polynomial, but $P_{\text{low}}$ is typically just a real polynomial.

Applying Theorem \ref{thm:squarefree}, there is a squarefree integer polynomial $Q$ of degree $n -m$ satisfying $\normff{Q}{n-m} \le n^{C_0 \sqrt{n}}$.  When we apply Corollary \ref{cor:adjust} to the real polynomial $\tfrac{1}{4}P_{\text{low}} +\tfrac{1}{2}$  with integer polynomial $Q$, we are left with an integer polynomial $R$ satisfying
\[\normff{\left(R - \tfrac{1}{4}P_{\text{low}} -\tfrac{1}{2}\right)}{n - m} \le n^{C_0 + C_0\sqrt{n}}.\]
Taking
\[R_n(z) = z^n + a_{n-1}z^{n-1} + \dots + a_{n-m}z^{n-m} + 4R(z) - 2,\]
we thus have
\begin{equation}
\label{eq:4nC0}
\normff{\left(R_n  - P_n\right)}{n - m} \le 4n^{C_0 + C_0\sqrt{n}}.
\end{equation}

Now, given a root $\alpha_k$ of $P_n$, take $[x_k, y_k]$ to be the subinterval constructed around $\alpha_k$ in Proposition  \ref{prop:early_ints_better} (4). Then \eqref{eq:big_strong_interval} holds with $C = C_1$. From \eqref{eq:main_n_ass} and \eqref{eq:4nC0}, we can conclude
\begin{equation}
\label{eq:final_triangle}
\left|P_n(x_k)\right| >\left|P_n(x_k) - R_n(x_k)\right| \quad\text{and}\quad \left|P_n(y_k)\right| >\left|P_n(y_k) - R_n(y_k)\right|.
\end{equation}
Since the interval $[x_k, y_k]$ contains exactly one root of $P_n$, $P_n(x_k)$ and $P_n(y_k)$ have opposite signs. From \eqref{eq:final_triangle}, $R_n(x_k)$ and $R_n(y_k)$ have opposite signs as well, so $[x_k, y_k]$ contains at least one root of $R_n$ for each $k \le n$. Since $R_n$ has degree exactly $n$, we find that each interval contains exactly one root and that $R_n$ has no roots outside the union of these intervals. So the roots of $R_n$ all lie in $\Sigma$.

The integer polynomial $R_n$ is irreducible by Eisenstein's criterion at $2$. Furthermore, from \eqref{eq:early_ints_normf} and \eqref{eq:4nC0}, we have
\[\normff{R_n}{n-m} \le 4n^{C_0 + C_0\sqrt{n}} + n^{C_1}\kappa^m.\]
Taking $C_2$ to be $\max_{x \in \Sigma} w_{\mu}(x)$, we may conclude
\[\normf{R_n}  \le (C_2\kappa)^{\lfloor B \sqrt{n}\log n\rfloor} \left(4n^{C_0 + C_0\sqrt{n}} + n^{C_1}\right).\]
So
\[\lim_{n \to \infty} \normf{R_n}^{1/n} = 1,\]
and Theorem \ref{thm:main} follows from an application of Lemma \ref{lem:RSphere}.
\qed

\subsection{Towards the proof of Proposition \ref{prop:early_ints_better}}
\label{ssec:early_ints}
For the rest of this section, we fix $\Sigma$ of capacity $\kappa > 1$ and $\mu$ as in the proposition statement, and we take $k_0$ to be the number of components of $\Sigma$. The proof of Proposition \ref{prop:early_ints_better} takes a similar approach to Robinson's proof \cite{Robi64, Serre18} that infinitely many algebraic integers have all conjugates inside $\Sigma$. The main strategy of both our proof and Robinson's proof is to take a Chebyshev polynomial for $\Sigma$ and nudge its coefficients to satisfy the conditions we want.

\begin{lem}[Chebyshev polynomials]
\label{lem:Chebyshev}
For $r \ge 0$, there is a unique real monic polynomial $T_r$ of degree $r$ for which there are $r+1$ points $x_0 < x_1 < \dots < x_r$ in $\Sigma$ such that
\[(-1)^{r - i}T_r(x_i) =  \normz{T_r} \quad\text{for all } \, i \le r.\]
This polynomial satisfies
\[ C^{-1} \cdot \kappa^r\le \normz{T_r} \le C \cdot \kappa^r,\]
where $C> 1$ depends just on $\Sigma$.
\end{lem}

\begin{proof}
This characterization of the Chebyshev polynomials for $\Sigma$ is known as the alternation theorem and can be found in \cite[Theorem 1.1]{CSZ_Cheb17}. The bounds on $\normz{T_r}$ follow from asymptotic work of Widom \cite[Theorem 11.5]{Widom69}; see also \cite{Totik09}.
\end{proof}

Because the intervals $[x_k, y_k]$ in Proposition \ref{prop:early_ints_better} (4) need to lie in $\Sigma$, it will be useful to work with polynomials without roots close to the boundary of $\Sigma$. One way to accomplish this is with the following definition.
\begin{defn}
\label{defn:prune}
Take $n$ to be a positive integer, and take $P$ to be a real monic squarefree polynomial. Take $S$ to be the set of roots of $P$. We assume there is a cardinality $n$ subset $\widetilde{S}$ of $S \cap \Sigma$ so that, for every connected component $I$ of $\Sigma$ that meets $S$, the intersection $I \cap \widetilde{S}$ contains neither the least nor greatest element of $I \cap S$. Choosing any $\widetilde{S}$ satisfying this condition, we then call the polynomial
\[\widetilde{P}(z) = \prod_{\alpha \in \widetilde{S}} (z - \alpha)\]
a \emph{degree $n$ pruned polynomial} of $P$.
\end{defn}

\begin{lem}
\label{lem:adjusted_Cheb}
There is $C > 0$ just depending on $\Sigma$ so, for every integer $r \ge 2$, there is a real monic degree $r$ polynomial $\widetilde{T}_r$ satisfying 
\[\normz{\widetilde{T}_r } \le r^C \cdot \kappa^r,\]
 an integer $k_1 \le  k_0$, and a sequence of real numbers $y_1< \dots < y_{r+k_1}$ lying in $\Sigma$ such that
\[ |\widetilde{T}_r(y_i)| \ge C^{-1} \cdot \kappa^r \quad\text{for } \, 1 \le i \le r + k_1,\]
such that $(y_1, y_{r+k_1})$ contains all the roots of $\widetilde{T}_r$, and such that $(y_i, y_{i+1})$ either does not lie entirely in $\Sigma$ and contains no root of $\widetilde{T}_r$, or lies entirely in $\Sigma$ and contains a unique root of $\widetilde{T}_r$ for each $i < r + k_1$. 
\end{lem}
\begin{proof}
From the alternation theorem, we see that $T_{r + 3k_0}$ has fewer than $k_0 $ roots outside $\Sigma$, so we may define $\widetilde{T}_r$ to be a degree $r$ pruned polynomial of $T_{r + 3k_0}$. Consider the real numbers $x_0 , \dots, x_{r + 3k_0}$ constructed in Lemma \ref{lem:Chebyshev}. We define an equivalence relation on these roots by placing $x_i, x_j$ in the same class if they lie in the same connected component of $\Sigma$ and there is no root of $\widetilde{T_r}$ between them. There are  at most $r + k_0$ such classes, and we take $y_1 < \dots < y_{r + k_1}$ to be a sequence of representatives for them. The intervals $(y_i, y_{i+1})$ then satisfy the final condition of the lemma.

From Lemma \ref{lem:Remez}, we have
\[\log \normz{\widetilde{T}_r} - \log \normz{T_{r + 3k_0}} \ll \log r,\]
where the implicit constants here and throughout the proof depend just on $\Sigma$. From Lemma \ref{lem:Chebyshev}, we thus have $\log \normz{\kappa^{-r} \widetilde{T}_r} \ll \log r$. This establishes the claimed bound for $\normz{\widetilde{T}_r}$. Meanwhile, for $i \le r + k_1$, we have
\[ \left|\widetilde{T}_r(y_i)\right| \gg \left|T_{r + 3k_0}(y_i)\right|\]
since $T_{r + 3k_0}/\widetilde{T}_r$ is a monic polynomial with roots contained in $\Sigma$. So Lemma \ref{lem:Chebyshev} gives
\[\left|\widetilde{T}_r(y_i)\right| \gg  \kappa^{r},\]
finishing the proof of the lemma.
\end{proof}
At this point, fix $C_1 > 0$  so Lemmas \ref{lem:Chebyshev} and \ref{lem:adjusted_Cheb} hold for $C_1 = C$, and take
\[D = \left\lceil \frac{4C_1^2 }{\kappa - 1}\right\rceil.\]

This integer is used in the following lemma, which follows the basic idea of \cite[Construction 7.3]{BCLPS21}. The same idea can also be seen in earlier work of Fekete and Szeg\H{o} \cite[Theorem D]{FeSz55}.

\begin{lem}
\label{lem:Bclps}
Take $P$ to be a monic real polynomial of degree $n$, and choose an integer $r \ge 2$.  Take $y_1 < \dots < y_{r + k_1}$ to be the sequence of points constructed from $\widetilde{T}_r$ in Lemma \ref{lem:adjusted_Cheb}. Then there is a monic real polynomial $T$ of degree $r$ so that the degree $i$ coefficient of
\[T^{D} P\]
is an even integer for all integers $i$ in $[(D -1)r + n, Dr + n - 1]$, and so that
\[T(y_i) \big/\widetilde{T}_r(y_i) \ge 1/2\]
for $1 \le i \le r+ k_1$. This polynomial also satisfies
\begin{equation}
\label{eq:T_normz}
\normz{T}  \le r^{C_1} \cdot \kappa^r + 2C_1 r \cdot \kappa^r.
\end{equation}
\end{lem}

\begin{proof}
Define a sequence of coefficients $\lambda_{r-1}, \lambda_{r-2}, \dots, \lambda_0$ as follows:
\begin{itemize}
\item Take $\lambda_{r-1} \ge 0$ minimal so
\[\left(\widetilde{T}_r + \lambda_{r-1} T_{r-1}\right)^{D} P\]
has an even coefficient at degree $Dr + n - 1$.
\item Take $\lambda_{r-2} \ge 0$ minimal so
\[\left(\widetilde{T}_r + \lambda_{r-1} T_{r-1} + \lambda_{r-2} T_{r-2}\right)^{D} P\]
has an even coefficient at degree $Dr + n-2$. Note that this does not affect the coefficient at degree $Dr + n-1$ of this product.
\item Repeat this process for $\lambda_{r-3}, \lambda_{r-4}, \dots, \lambda_0$.
\end{itemize}
We then take
\[T = \widetilde{T}_r + \lambda_{r-1} T_{r-1} + \dots + \lambda_1 T_1 + \lambda_0,\]
and we see that $T^{D}P$ has even coefficients in degrees lying in  $[(D -1)r + n, Dr + n - 1].$

The $\lambda_i$ are bounded by $2/D$ since $P$ is monic, so
\[\lambda_i \left|T_{r - i}(y_j)\right|  \le 2C_1 \cdot \kappa^{r-i}D^{-1}\quad\text{ for }\, 1 \le j \le r + k_1\,\,\text{ and }\,\,1 \le  i \le r.\]
by  Lemma \ref{lem:Chebyshev}. So
\[ \left|\sum_{i =1}^{r} \lambda_{r -i} T_{r-i}(y_j)\right| \le \kappa^r \sum_{i = 1}^r 2C_1 \kappa^{-i} D^{-1} \le \frac{2C_1\kappa^r}{D(\kappa - 1)} \le \tfrac{1}{2}C_1^{-1}\kappa^r,\]
with the last inequality following from the choice of $D$. So Lemma \ref{lem:adjusted_Cheb} gives that $T(y_j)/\widetilde{T}_r(y_j)$ is at least $1/2$ for $j$ in $[1, r + k_1]$.

Finally, $\normz{\widetilde{T}_r}$ is at most $r^{C_1} \cdot \kappa^r$ by Lemma \ref{lem:adjusted_Cheb}, and $\normz{\lambda_i T_i}$ is at most $2C_1 \kappa^r$ for $i \le r$ by Lemma \ref{lem:Chebyshev}, so \eqref{eq:T_normz} follows.
\end{proof}
We now introduce the central construction used in Proposition \ref{prop:early_ints_better}.
\begin{defn}
\label{defn:comp_Q}
Take $\Sigma$, $\kappa$, $C_1$, and $D$ as above. Choose a real monic polynomial $P$ of degree $n$ and an integer $r \ge 2$. We then define a \emph{complementary polynomial to $P$ of degree $Dr$} as follows.

To start, take $T$ to be the polynomial constructed in Lemma  \ref{lem:Bclps} from $P$. We then take
\begin{equation}
\label{eq:Q0_def}
Q_0(x) = \prod_{j = 0}^{D-1} \left(T(x)  - \frac{2j + 1}{4DC_1} \kappa^r\right).
\end{equation}
Taking $d = rD - r$, we see that $Q_0 - T^D$ has degree at most $d$. Since $T^DP$ had even coefficients in degrees lying in $[d + n + 1, Dr + n -1]$, it follows that $Q_0P$ also has even coefficients in these degrees.

We define a sequence $\epsilon_{d}, \dots, \epsilon_1$ of real numbers as follows:
\begin{itemize}
\item Take $\epsilon_{d} \ge 0$ minimal so
\[ \left(Q_0 + \epsilon_{d}T_{d}\right)\cdot P\]
 has an even coefficient at degree $d + n$.
\item Take $\epsilon_{d-1} \ge 0$ minimal so
\[\left(Q_0 + \epsilon_{d} T_{d} + \epsilon_{d - 1}T_{d-1}\right)\cdot P\]
 has an even coefficient at degree $d + n- 1$.
\item Repeat this process for $\epsilon_{d -2}, \dots, \epsilon_1$.
\end{itemize}
Take
\[C_2 = 8\cdot(4DC_1)^D\]
We then take
\[Q = Q_0 +\epsilon_{d} T_{d} + \epsilon_{d-1}T_{d-1} + \dots  + \epsilon_1 T_1+ N,\]
where $N$ is the minimal nonnegative real number for which this polynomial satisfies
\begin{equation}
\label{eq:pigeonhole_space}
|Q(\alpha)| \ge n^{-1}C_2^{-1}\kappa^{Dr} \quad\text{for every root } \alpha \text{ of } P.
\end{equation}
We define the complementary polynomial to $P$ of degree $Dr$ to be this polynomial $Q$. From the construction, it is clear that the degree $i$ coefficient of $PQ$ is an even integer for $i$ in $[n + 1, n + Dr - 1]$.
\end{defn}

\begin{lem}
\label{lem:comp_Q_roots}
Given $\Sigma$, $\kappa$, $C_1$, and $D$ as above, there is $C > 1$ so we have the following.

Choose integers $r \ge C$ and $n \ge 0$, and choose a real monic polynomial $P$ with degree $n$. Take $Q$ to be the degree $Dr$ complementary polynomial to $P$.

Then $Q$ is squarefree, and all its roots lie in $\Sigma$. Furthermore, given any root $\beta$ of $Q$, there is an interval $[x, y]$ contained in $\Sigma$ and containing $\beta$ so no other root of $Q$ lies in $[x, y]$ and so
\[|Q(x)|, |Q(y)| \ge C^{-1} \cdot \kappa^{Dr}.\]
\end{lem}
\begin{proof}
As part of the construction of $Q$, we used the degree $r$ polynomial $T$ constructed in Lemma \ref{lem:Bclps}. Take $y_1 < \dots < y_{r + k_1}$ to be the sequence of real numbers constructed with $\widetilde{T}_{r}$ in Lemma \ref{lem:adjusted_Cheb}.

Take $S$ to be the set of integers $i$ in $[1, r+k_1)$ for which $(y_i, y_{i+1})$ contains a root of $T$. By Lemmas \ref{lem:adjusted_Cheb} and \ref{lem:Bclps}, this set has cardinality $r$. For $i$ in $S$,  the root of $T$ in the interval $[y_i, y_{i+1}]$ is unique, so $T(y_i)$ and $T(y_{i+1})$ have opposite signs. From Lemmas \ref{lem:adjusted_Cheb} and \ref{lem:Bclps}, we have
\[\min\big(|T(y_i)|, \,|T(y_{i+1})|\big) \ge \tfrac{1}{2}C_1^{-1}\kappa^r.\]
For any integer $j$ in $[0, D]$, the intermediate value theorem lets us choose a point $w_{i,j}$ in $[y_i, y_{i+1}]$ for which
\[T(w_{i,j}) = \frac{j}{2DC_1}\kappa^r.\]
We take $v_{i,0}< \dots < v_{i,D}$ to be the ordered elements in $\{w_{i,0}, \dots, w_{i,D}\}$.

We now consider the polynomial $Q_0$ defined in \eqref{eq:Q0_def}.  Given integers $j_1, j_2$ in $[0, D]$ with $j_1 < j_2$, we see there is an integer $j$ in $[0, D-1]$ so
\[\frac{j_1}{2 DC_1} \kappa^r < \frac{2j + 1}{4DC_1}\kappa^r < \frac{j_2}{2DC_1}\kappa^r.\]
Again by the intermediate value theorem, we thus see that $Q_0$ has a root in every interval $(v_{i, j}, v_{i,  j+1})$ for $j$ in $[0, D-1)$ and $i$ in $S$. Since $Q_0$ has degree $Dr$, the root in this interval must be unique, so $Q_0(v_{i, j})$ and $Q_0(v_{i, j+1})$ have opposite signs. Our goal is to show that the same holds for $Q$.

Take $N$ and $\epsilon_1, \dots, \epsilon_d$ as in Definition \ref{defn:comp_Q}. We will start by bounding $N$. If we take
\[a =  2C_2^{-1}n^{-1}\kappa^{Dr},\]
we see that, for any root $\alpha$ of $P$, $|(Q - ka)(\alpha)|$ is less than $n^{-1}C_2^{-1}\kappa^{Dr}$ for at most one integer $k$. Since $N$ was chosen to be minimal, we also know that $|(Q-ka)(\alpha)|$ is less than this bound for some $\alpha$ for all $k$  satisfying $0 \le k \le N/a$. By the pigeonhole principle, we thus have
\begin{equation}
\label{eq:Nbound}
N \le na = 2C_2^{-1} \kappa^{Dr}.
\end{equation}

Next, the $\epsilon_i$ appearing in the construction are at most $2$. From Lemma \ref{lem:Chebyshev}, we have
\begin{equation}
\label{eq:eps_bound}
\left| \sum_{i = 1}^d\epsilon_i T_i(x)  \right| \le \sum_{i = 1}^d 2C_1 \cdot \kappa^i \le \frac{2C_1 \kappa^{d+1}}{\kappa - 1} = \frac{2C_1 \kappa^{-r +1}}{\kappa - 1} \cdot \kappa^{Dr}.
\end{equation}
for all $x \in \Sigma$. Finally, the definition of $v_{i, j}$ and \eqref{eq:Q0_def} give
\[|Q_0(v_{i, j}) |\ge (4DC_1)^{-D} \kappa^{Dr} = 8C_2^{-1}\kappa^{Dr}.\]
From \eqref{eq:Nbound} and \eqref{eq:eps_bound}, we have. 
\[|Q(v_{i, j}) - Q_0(v_{i, j})| \le 2C_2^{-1}\kappa^{Dr} + \frac{2C_1 \kappa^{-r +1}}{\kappa - 1} \cdot \kappa^{Dr}.\]
At this point, select $C > 1$ so that
\[6C_2^{-1} - \frac{2C_1 \kappa^{-r +1}}{\kappa - 1} \, > \, C^{-1} \quad\text{for all } r \ge C .\]
The previous inequalities then give
\[|Q_0(v_{i, j})| - |Q(v_{i, j}) - Q_0(v_{i, j})| >C^{-1} \cdot \kappa^{Dr} \quad\text{ if } r \ge C.\]
We may conclude that $Q(v_{i, j})$ shares the sign of  $Q_0(v_{i, j})$ and has magnitude at least $C^{-1} \cdot \kappa^{Dr}$ if $r \ge C$.

Now assuming $r\ge C$, we see that $Q(v_{i, j})$ and $Q(v_{i, {j+1}})$ have different signs for $j < D$ and $i \in S$ since $Q_0$ had this property. Applying this for all $i$ in $S$ and $0 \le j < D$, we see that $Q$ has a unique root in the interval $(v_{i, j}, v_{i, j+1})$ and that each root of $Q$ is in such an interval. Given a root $\beta$ of $Q$, we thus may take $x = v_{i, j}$ and $y = v_{i, j+1}$, where $(i, j)$ is chosen so $\beta$ lies in $(v_{i, j}, v_{i, j+1})$.  The result follows. 
\end{proof}

We collect the other properties we need of the complementary polynomials now.
\begin{lem}
\label{lem:comp_Q_roots2}
Given $\Sigma$, $\kappa$, $C_1$, and $D$ as above, there is $C > 1$ so we have the following.

Choose an integer $r \ge C$, and choose a real monic polynomial $P$. Take $Q$ to be the degree $Dr$ complementary polynomial to $P$.

Then 
\begin{equation}
\label{eq:QQprime}
\normz{Q} \le r^C \kappa^{Dr} \quad\text{and}\quad \normz{Q'} \le r^C \kappa^{Dr}.
\end{equation}
Furthermore, given a root $\beta$ of $Q$, we have
\[|Q(x)| \ge C^{-1}\kappa^{Dr} |x - \beta| \quad\text{ for all } x \in \left[\beta - r^{-C},\, \beta + r^{-C}\right].\]
\end{lem}
\begin{proof}
Take $Q_0$ and $T$ as in Definition \ref{defn:comp_Q}. Then $\log\normz{\kappa^{-r}T} \ll \log r$ by Lemma \ref{lem:Bclps}. Here and throughout this proof, the implicit constants depend just on $\Sigma$, $\mu$, and $C_1$; we recall that $D$ and $C_2$ just depend on these choices.

It then follows from \eqref{eq:Q0_def} that $\log \normz{\kappa^{-Dr} Q_0} \ll \log r$. From \eqref{eq:eps_bound} and \eqref{eq:Nbound}, we have $\log \normz{\kappa^{-Dr} (Q - Q_0)} \ll \log r$, so
\begin{equation}
\label{eq:normzQ}
\log \normz{\kappa^{-Dr}Q} \ll \log r.
\end{equation}
Given a root $\beta$ of $Q$, take $Q_{\beta}$ to be the polynomial $Q(z)/(z - \beta)$. Then
\[\normz{Q'} = \normz{\sum_{\beta} Q_{\beta}} \le \,\max_{\beta} Dr \cdot \normz{Q_{\beta}},\]
where the sum and maximum are over the roots of $Q$. Lemma \ref{lem:Remez} gives
\[\log \normz{Q_{\beta}} - \log\normz{Q} \ll \log r.\]
We thus have 
\[\log\normz{Q'} - \log \normz{Q} \ll \log r.\]
Together with \eqref{eq:normzQ}, this implies \eqref{eq:QQprime} for $C$ sufficiently large

Now choose a root $\beta$ of $Q$. Supposing $r$ is sufficiently large, choose $[x, y]$ satisfying the condition of Lemma \ref{lem:comp_Q_roots} for $\beta$. We claim that
\begin{equation}
\label{eq:headache}
\min\big(|Q_{\beta}(x)|, \,|Q_{\beta}(y)|\big) \le |Q_{\beta}(w)| \text{ for all } w \in (x, y).
\end{equation}

To start, take $\beta_1 < \dots < \beta_{Dr-1}$ to be the roots of $Q_{\beta}$. These are distinct elements of $\Sigma$ outside of $[x, y]$. By Rolle's theorem, we see that $Q_{\beta}'$ must be $0$ somewhere on each interval $(\beta_i, \beta_{i+1})$  for $i <D r-1$. Since $Q_{\beta}'$ has degree $Dr-2$, its root on $(\beta_i, \beta_{i+1})$ is unique, and it has no roots outside $(\beta_1, \beta_{Dr-1})$.

We now prove \eqref{eq:headache}. There are two cases. First, if $(x, y)$ contains a root of $Q'_{\beta}$, then this interval must lie in $[\beta_i, \beta_{i+1}]$ for some $i < Dr-1$. If \eqref{eq:headache} does not hold for some $w$, we find that $|Q_{\beta}|$ has a local maximum in both $(\beta_i, w)$ and $(w, \beta_{i+1})$. But $|Q_{\beta}|$ is a smooth, nonnegative function on $[\beta_i, \beta_{i+1}]$, and $Q_{\beta}'$ is $0$ only once in this interval, so this cannot happen.

Otherwise, $Q'_{\beta}$ has no root in $(x, y)$, so $Q_{\beta}$ is monotonic in this interval, and the claim follows since $Q_{\beta}(x)$ and $Q_{\beta}(y)$ have the same sign.

From Lemma \ref{lem:comp_Q_roots}, we have
\begin{equation}
\label{eq:Qxboundagain}
\min\big(|Q(x)|, |Q(y)| \big)\gg  \kappa^{Dr}.
\end{equation}
From the fundamental theorem of calculus and our bounds on $|Q'|$, this implies
\[\min\big(\log|x - \beta|, \,\log |y  - \beta|\big) \gg - \log r.\]
So $[\beta - r^{-C}, \,\beta + r^{-C}]$ is a subset of $[x, y]$ for $C$ sufficiently large.

We have $|x - \beta| \ll 1$ and  $|y - \beta| \ll 1$ since $\Sigma$ is compact, so \eqref{eq:Qxboundagain} also gives
\[\min\big(|Q_{\beta}(x)|, |Q_{\beta}(y)| \big)\gg \kappa^{Dr}.\]
From \eqref{eq:headache}, we thus have
\begin{equation}
\label{eq:headache_real}
|Q_{\beta}(w)| \gg \kappa^{Dr}\quad\text{for }\, w \in [x, y] \supseteq [\beta - r^{-C}, \,\beta + r^{-C}]
\end{equation} 
for $C$ sufficiently large. The lemma follows.
\end{proof}

\begin{lem}
\label{lem:P_ver1}
With $\Sigma$ and $\mu$ fixed as above, and given an integer $d$, there is  $C > 0$ so we have the following:

For every integer $n \ge 2$, there is a monic squarefree polynomial $\widetilde{P}_n$ so
\begin{enumerate}
\item For every root $\alpha$ of $\widetilde{P}_n$, the interval $[\alpha - n^{-C}, \alpha + n^{-C}]$ is contained in $\Sigma$ and contains no root of $\widetilde{P}_n$ besides $\alpha$.
\item We have
\[  w^{n+d}_{\mu}(x) \left|\widetilde{P}_n(x)\right| \ge n^{-C} \min_{\alpha}|x - \alpha| \quad\text{for all } x \in \Sigma,\]
where the minimum is taken over all roots of $\widetilde{P}_n$.
\item We have $\normff{\widetilde{P}_n}{n+d} \le n^C$.
\end{enumerate}
\end{lem}
\begin{proof}
Given $n \ge 2$, take $P = P_{n + 2k_0, \mu}$ to be the degree $n+2k_0$ approximating polynomial for $\mu$ as given in Definition \ref{defn:approx}.  By Lemma \ref{lem:pot_Hold}, $U^{\mu}$  is bounded, so
\begin{equation}
\label{eq:d2k0bnd}
\big|\log w_{\mu}^{d - 2k_0}(x)\big| \ll 1 \quad\text{for }\, x \in \Sigma.
\end{equation}
Here and throughout this proof, the implicit constant for $\ll$ depends just on $\Sigma$, $\mu$, and $d$.  We have 
\begin{equation}
\label{eq:xyll}
|x - y| \ll 1\quad\text{ for all }\, x, y \in \Sigma
\end{equation}
since $\Sigma$ is compact, so Proposition  \ref{prop:three_pages} gives $\log \normff{P}{n+2k_0} \ll \log n$. From \eqref{eq:d2k0bnd}, we thus have
\[ \log \normff{P}{n + d} \ll \log n.\]
Take $\widetilde{P}_n$ to be a degree $n$ pruned polynomial for $P$.  This polynomial is monic and squarefree. We claim this satisfies the conditions of the lemma.

First, from \eqref{eq:Holder_spaced}, we have $\log| \alpha - \alpha'| \gg - \log n$ for any two roots of $P$. The definition of a pruned polynomial then gives (1).

Next, Lemma \ref{lem:Remez} gives
\[\log\normff{\widetilde{P}_n}{n+d} - \log\normff{P}{n+d} \ll \log n,\]
so $\log \normff{\widetilde{P}_n}{n+d} \ll \log n$. This gives (3).

This leaves (2). Choose $x \in \Sigma$, and choose a root $\alpha$ of $P$ for which $|x - \alpha|$ is minimized. We assume $x \ne \alpha$. Proposition \ref{prop:three_pages} gives
\[\log\left| w_{\mu}^{n + d}(x) P(x)\right| - \log|x - \alpha| \gg - \log n.\]
Suppose first that $\alpha$ is a root of $\widetilde{P}_n$. Taking $Q(x)$ to be the polynomial $P(x)/\widetilde{P}_n(x)$, we have  $\log|Q(x)| \ll 1$ for $x \in \Sigma$ by \eqref{eq:xyll}. So
\[\log\left| w_{\mu}^{n + d}(x) \widetilde{P}_n(x)\right| - \log|x - \alpha| \gg - \log n.\]
Otherwise, take $Q(x)$ to be the polynomial $P(x)/\big((x - \alpha) \cdot \widetilde{P}_n(x)\big)$. We still have $\log|Q(x)| \ll 1$, implying $\log\left| w_{\mu}^{n + d}(x) \widetilde{P}_n(x)\right| \gg - \log n$, and applying \eqref{eq:xyll} gives
\[\log\left| w_{\mu}^{n + d}(x) \widetilde{P}_n(x)\right| - \log|x - \alpha'| \gg - \log n\]
for any root $\alpha'$ of $\widetilde{P}_n$. From \eqref{eq:d2k0bnd}, we find that (2) holds except potentially in the case when $x$ is a root of $P$.  We may conclude the part from the continuity of both sides of the inequality.
\end{proof}

\begin{proof}[Proof of Proposition \ref{prop:early_ints_better}]
Take $\mu$, $\Sigma$, and $B$ as in the proposition statement, and define $C_1$ and $D$ from $\mu$ and $\Sigma$ as above. Choose $C_0 > 1$ such that Lemmas \ref{lem:comp_Q_roots} and \ref{lem:comp_Q_roots2} hold for $\Sigma$  with $C = C_0$, such that Lemma \ref{lem:Remez}  holds for $(\Sigma, \mu)$ with $C = C_0$, and such that Lemma \ref{lem:P_ver1} holds for $ (\Sigma, \mu)$ and all integers  $d$ in $[1, D]$ with $C = C_0$. We also assume that $\Sigma$ is contained in $[-C_0, C_0]$.

Given a sufficiently large integer $n$, we construct $P_n$ as follows. First, we take
\[m = \lfloor B \sqrt{n}\log n\rfloor\quad\text{and}\quad r = \left\lceil \frac{m+1}{D}\right\rceil.\]
We will assume that $n$ is large enough that $n - Dr \ge 2$ and $r \ge 2$. We then take $\widetilde{P}_{n - Dr}$ to be the polynomial constructed in Lemma \ref{lem:P_ver1} for $\Sigma$ and $\mu$  of degree $n - Dr$, and we take $Q_{Dr}$ to be a degree $Dr$ complementary polynomial to $\widetilde{P}_{n-Dr}$. We then take
\[P_n = \widetilde{P}_{n - Dr}Q_{Dr}\]
By the definition of a complementary polynomial,  this has even coefficients for all degrees in $[n-Dr + 1, n-1]$. Since $Dr \ge m +1$, we see that this polynomial obeys condition (3) of the proposition.

Both $\widetilde{P}_{n-Dr}$ and $Q_{Dr}$ have all roots in $\Sigma$, and they share no roots by the assumption \eqref{eq:pigeonhole_space} made as part of the definition of a complementary polynomial. So these polynomials obey condition (1) of the proposition.

For condition (2), we have
\[\normz{Q_{Dr}} \le r^{C_0} \kappa^{Dr} \quad\text{and} \quad \normff{\widetilde{P}_{n - Dr}}{n - m} \le (n-Dr)^{C_0}\]
from Lemmas \ref{lem:comp_Q_roots2} and \ref{lem:P_ver1}, with the latter being applied with $d = Dr - m$. Taking the product, we have
\[\normff{P_n}{n - m} \le n^{2C_0}\kappa^{Dr} \le n^{2C_0} \kappa^D \kappa^m.\]
So condition (2) holds for sufficiently large $C$.

So we now just need to show that condition (4) holds for a good choice of $C$. Define the real number $C_2$ as in Definition \ref{defn:comp_Q}, and take
\[\delta = \min\left(\tfrac{1}{3}n^{-C_0}, \tfrac{1}{3}n^{-C_0 - 1}C_2^{-1}\right).\]
We claim that any two distinct roots $\alpha_1, \alpha_2$ of $P_n$ satisfy $|\alpha_1 - \alpha_2| \ge 3\delta$. This is clear from Lemmas \ref{lem:comp_Q_roots2} and \ref{lem:P_ver1} unless one of the roots is of $\widetilde{P}_{n-Dr}$ and the other is of $Q_{Dr}$, so we focus on this case. Take $\alpha$ to be a root of $\widetilde{P}_{n - Dr}$. Lemma \ref{lem:comp_Q_roots2} and \eqref{eq:pigeonhole_space} then give
\[\normz{Q'_{Dr}} \le n^{C_0} \kappa^{Dr}\quad\text{and}\quad |Q_{Dr}(\alpha)| \ge n^{-1}C_2^{-1}\kappa^{Dr}.\]
As a result, we have
\[|Q_{Dr}(x)| \ge n^{C_0}\kappa^{Dr}\cdot \left(n^{-C_0 - 1}C_2^{-1} - |x - \alpha|\right) \ge n^{C_0}\kappa^{Dr}\cdot \left(3 \delta - |x - \alpha|\right).\]
for all $x \in \Sigma$. So $\alpha$ must have distance at least $3\delta$ from any root of $Q_{Dr}$.

Now choose a root $\alpha$ of $P_n$, and take $x$ to be either $\alpha + \delta$ or $\alpha - \delta$. From this last formula, we have
\[|Q_{Dr}(x)| \ge n^{C_0} \kappa^{Dr} \cdot2\delta\]
if $\alpha$ is a root of $\widetilde{P}_{n-Dr}$. Otherwise, $\alpha$ is a root of $Q_{Dr}$, and Lemma \ref{lem:comp_Q_roots2} implies
\[|Q_{Dr}(x)| \ge n^{-C_0}\kappa^{Dr}\cdot \delta.\]
In either case, Lemma \ref{lem:P_ver1} gives
\[w^{n - m}_{\mu}(x) \cdot \left|\widetilde{P}_{n - Dr}(x)\right| \ge n^{-C_0} \delta.\]
So, taken together, we have
\[w^{n - m}_{\mu}(x)  \left|P_n(x)\right| \ge n^{-2C_0} \kappa^{Dr} \delta^2.\]
This is at least $n^{-C}\kappa^{Dr} \ge n^{-C} \kappa^m$ so long as
\[ n^{-C} < n^{-2C_0} \min\left(\tfrac{1}{3}n^{-C_0}, \tfrac{1}{3}n^{-C_0 - 1}C_2^{-1}\right)^2 \quad\text{for all } n > C,\]
which does hold for sufficiently large $C$.  

Suppose $C$ has been chosen large enough that this holds, and suppose that $n > C$. Then, taking  $\alpha_1 < \dots < \alpha_n$ to be the roots of $P_n$, we may take $x_i = \alpha_i - \delta$ and $y_i = \alpha_i + \delta$ for each $i \le n$ to satisfy condition (4) of the proposition. The proposition follows.
\end{proof}

\section{Limits of measures and Serre's example}
\label{sec:limits}
The first goal of this section is to prove Proposition \ref{prop:limit_Holder}, which reduced the proof of  Theorem \ref{thm:main} to its proof for H\"{o}lder measures on a compact finite union of intervals. As part of this, the following measure will be important.
\begin{notat}
\label{notat:balbump}
Given real numbers $b> a> 0$, we define $\nu_{[a, b]}$ to be the measure supported on $[a, b]$ given by
\[d\nu_{[a, b]}(t) = \frac{dt\sqrt{ab}}{\pi t \sqrt{(b - t)(t - a)}}.\]
This measure is $1/2$-H\"{o}lder.
\end{notat}

\begin{lem}
\label{lem:balay}
Given any $b > a > 0$, we have the following:
\begin{enumerate}
\item The measure $\nu_{[a, b]}$ is a probability measure. 
\item For all $z \in \C$, we have
\begin{equation}
\label{eq:first_balay}
U^{\nu_{[a, b]}}(z) + \log|z| - \log\left(\frac{a + 2\sqrt{ab} + b}{b - a}\right) = \begin{cases} 0 &\text{ if } z \in [a, b] \\ \le 0 &\text{ otherwise.}\end{cases}
\end{equation}
\item We have
\[U^{\nu_{[a, b]}}(0) = \log\left(\frac{a + 2\sqrt{ab} + b}{4ab}\right).\]
\end{enumerate}
\end{lem}
\begin{proof}
Take $\nu = \nu_{[a, b]}$. By \cite[(II.4.47)]{SaTo97}, $\nu$ is the balayage of the probability measure $\delta_0$ supported on $\{0\}$ to the interval $[a, b]$. As such, it is automatically a probability measure, giving (1).

For the remaining parts, we follow the argument of \cite[Appendix B]{AgPe08}. Take $\mu$ to be the measure supported on $[b^{-1}, a^{-1}]$ given by 
\[d\mu(t) = \frac{dt}{\pi \sqrt{(a^{-1} - t)(t - b^{-1})}}.\]
As shown in \cite[Example I.3.5]{SaTo97}, this is the unweighted equilibrium measure on this interval, which has capacity $\tfrac{1}{4}(a^{-1} - b^{-1})$. In particular, it is a probability measure, so $\int d\mu = 1$, and we have
\begin{equation}
\label{eq:equibblib}
U^{\mu}(z) + \log\left(\frac{b - a}{4ab}\right) = \begin{cases} 0 &\text{ if } z \in [b^{-1}, a^{-1}] \\ \le  0 &\text{ otherwise}\end{cases}
\end{equation}
except potentially on a capacity $0$ subset of $[b^{-1}, a^{-1}]$ by \cite[(I.1.4) and (I.1.9)]{SaTo97}. The left hand side of this relation is a continuous function on $[b^{-1}, a^{-1}]$, so the equality must hold on $[b^{-1}, a^{-1}]$. From \cite[(10) in Appendix B]{AgPe08}, we also have
\begin{equation}
\label{eq:Serre10}
U^{\mu}(0) = -\log\left(\frac{a+ b + 2\sqrt{ab}}{4ab}\right).
\end{equation}

Given nonzero complex $z$, the $u$-subsitution  $u = t^{-1}$ gives
\[U^{\nu}(z) = \int_{a}^{b} - \log|z - t|  d\nu(t) =  \int_{b^{-1}}^{a^{-1}} -\log|z - u^{-1}|d\mu(u).\]
Since $-\log|z - u^{-1}| = -\log|z| + \log|u| - \log|z^{-1} - u|$ for  $u$ nonzero, we thus have
\begin{align}
\label{eq:nu_Serre}
U^{\nu}(z) \,&=\, - \log|z| \cdot \int  d\mu(u)  + \int \log|u| d\mu(u) - \int \log|z^{-1} - u| d\mu(u)\\
& =  \,- \log|z| - U^{\mu}(0) + U^{\mu}(z^{-1}).\nonumber
\end{align}
If $z$ lies in $[a, b]$, then $z^{-1}$ lies in $[b^{-1}, a^{-1}]$. From \eqref{eq:equibblib} and \eqref{eq:Serre10}, we find that \eqref{eq:first_balay} holds unless $z = 0$. But \eqref{eq:first_balay} takes the form $-\infty \le 0$ at $z = 0$, so (2) follows.

The function $U^{\nu}$ is continuous on $\R$ by  Lemma \ref{lem:pot_Hold}, and we have
\[\lim_{x \to 0} U^{\mu}(x^{-1}) - \log|x|  = 0\]
since $\mu$ is a probability measure with compact support. So $U^{\nu}(0) = -U^{\mu}(0)$ by \eqref{eq:nu_Serre}, giving (3).
\end{proof}

\begin{notat}
For $\epsilon$ in $(0, 1/2)$, take $\nu_{\epsilon}$ to be $\nu_{[\epsilon^2, \epsilon]}$. Given any other Borel measure $\mu$ with support contained in a compact subset $\Sigma$ of $\R$, we can consider the convolution measure $\mu * \nu_{\epsilon}$, which is defined by
\[\mu * \nu_{\epsilon}(Y) =  \int_{\R} \nu_{\epsilon}(Y - x) d\mu(x)\]
for any Borel set $Y$ in $\R$. From this formula, we see that
\[\mu * \nu_{\epsilon}\big([x_0, x_1]\big) \le \mu(\Sigma) \cdot \max_{t \in \R} \nu_{\epsilon}\big([x_0 + t, x_1 + t]\big)\]
for any real numbers $x_0 \le x_1$. Since $\nu_{\epsilon}$ is $1/2$-H\"{o}lder, it follows that this measure is also $1/2$-H\"{o}lder. Its support lies in
\[\big\{x \in \R\,:\,\, [x - \epsilon, x - \epsilon^2] \cap \Sigma \ne \emptyset\big\}.\]
For any $z \in \C$, we have
\[U^{\mu * \nu_{\epsilon}}(z) = \int \int -\log|z - x -t| d\mu(x)d\nu_{\epsilon}(t) = \int U^{\nu_{\epsilon}}(z - x) d\mu(x).\]
By Lemma \ref{lem:balay}, this is at most
\[ \log\left(\frac{\epsilon^2 + 2\epsilon^{3/2} + \epsilon}{\epsilon - \epsilon^2}\right)  - \int \log|z - x| d\mu(x),\]
and we can conclude that
\begin{equation}
\label{eq:good_imitation}
U^{\mu * \nu_{\epsilon}}(z) - U^{\mu}(z) \le   \log\left(\frac{\epsilon^2 + 2\epsilon^{3/2} + \epsilon}{\epsilon - \epsilon^2}\right)    \le C\epsilon^{1/2},
\end{equation}
where $C > 0$ does not depend on $z$, $\epsilon$, or $\mu$.

\end{notat}

\begin{lem}
\label{lem:weaks_smear}
Suppose there is $\epsilon_0 < 1/2$ so $\mu *\nu_{\epsilon}$ is supported in $\Sigma$ for all $\epsilon < \epsilon_0$.  Then $\mu * \nu_{\epsilon}$ converges to $\mu$ in the  \weaks topology as $\epsilon$ tends to $0$. 
\end{lem}
\begin{proof}

We first note that, given any continuous function $g : [0, \epsilon_0]\to \R$, we have
\begin{equation}
\label{eq:delta_weaks}
\lim_{\epsilon \to 0^+} \int_{\epsilon^2}^{\epsilon} gd\nu_{\epsilon} = g(0).
\end{equation}
since $\max_{x \in [\epsilon^2, \epsilon]} g(x)$ and $\min_{x \in [\epsilon^2, \epsilon]} g(x)$ tend to $g(0)$ as $\epsilon$ tends to $0$. 

Choose $C > 2$ so $\Sigma$ is contained in $[-C + 1, C - 1]$. Given a continuous function $f: \Sigma \to \R$, the Tietze extension theorem \cite[Theorem 15.8]{Will04} guarantees that there is a continuous function $f_1: [-C, C] \to \R$ restricting to $f$. This function is uniformly continuous since $[-C, C]$ is compact, so the function $g: [0, 1) \to \R$ given by
\[g(y) = \int_{-C + 1}^{C-1} f_1(x + y) d\mu(x)\]
is continuous. So \eqref{eq:delta_weaks} implies
\begin{align*}
\lim_{\epsilon \to 0^+} \int_{\Sigma} f d(\mu * \nu_{\epsilon}) &= \lim_{\epsilon \to 0^+} \int \int f(x + y) d\mu(x)d\nu_{\epsilon}(y) \\
&=\lim_{\epsilon \to 0^+} \int g d\nu_{\epsilon} = g(0) = \int f d\mu,
\end{align*}
establishing \weaks convergence.
\end{proof}

The following construction will also be useful.
\begin{defn}
\label{defn:sweet}
Choose a compact subset $\Sigma$ of $\R$ of capacity $\kappa > 1$, and choose a compact subset $\Sigma'$ of $\R$ containing $\Sigma$. Choose measures $\mu$ and $\nu$ with support contained in $\Sigma'$. We assume $\mu$ is a probability measure. Taking
\[ \gamma  =1 - \nu(\Sigma'),\]
we assume $\gamma$ lies in $[0, 1]$.
Take $\mu_{\Sigma}$ to be the unweighted equilibrium measure for $\Sigma$, and consider 
\begin{equation}
\label{eq:def_B}
B = \sup_{\substack{U^{\mu}(z) < \infty}} \left(U^{\nu + \gamma\cdot \mu_{\Sigma}}(z) - U^{\mu }(z)\right).
\end{equation}
We assume $B$ exists and is finite. Assuming this, it must be nonnegative, as $U^{\nu + \gamma\cdot \mu_{\Sigma}}(z) + \log|z|$ and  $U^{\mu }(z) + \log |z|$ both tend to $0$ as $|z|$ tends to $\infty$.

We then define the \emph{sweetened measure} of $\nu$ with respect to $\Sigma$ and $\mu$ to be the probability measure
\[\text{sw}(\nu) = (\beta + \gamma - \beta\gamma)\cdot  \mu_{\Sigma} + (1 - \beta)\cdot \nu,\]
where $\beta$ in $[0, 1)$ is selected so
\[\frac{\beta}{1 -\beta} = \frac{B}{\log \kappa}.\]
The \emph{measure of sweetener} in this measure is defined to be $\beta + \gamma - \beta \gamma$.

By \cite[(I.1.4)]{SaTo97} and the definition of $B$ and $\beta$, we have 
\begin{align}
\label{eq:sw_def}
U^{\textup{sw}(\nu)}(z) &\,=\, (\beta +\gamma - \beta\gamma) U^{\mu_{\Sigma}}(z) + (1 - \beta) U^{\nu}(z) \\
\nonumber
&\le\,-\beta \log \kappa + \gamma(1 - \beta)U^{\mu_{\Sigma}}(z) +  (1 -\beta) U^{\nu}(z)\\
\nonumber
& = \, - (1 - \beta)B + (1 - \beta)U^{\nu + \gamma \cdot \mu_{\Sigma}}(z)\\
\nonumber
 &\le\, (1 - \beta) U^{\mu}(z)
\end{align}
for all $z \in \C$.

Choose a nonconstant integer polynomial $Q$, and take its leading term to be $a$.
Given any probability measure $\mu'$ supported on $\Sigma'$, we have
\begin{align}
\label{eq:logQ_pot}
\int_{\Sigma'} \log|Q| d\mu' &=\,\log|a| +  \deg Q \int \int \log|z - x| d\mu'(x) d\mu_Q(z) \\
&= \,\log|a| - \deg Q \int  U^{\mu'}(z) d\mu_Q(z),\nonumber
\end{align}
where $\mu_Q$ is the associated counting measure for $Q$. From \eqref{eq:sw_def}, we have
\begin{align*}
\log|a| - \deg Q \int U^{\textup{sw}(\nu)}(z)d\mu_Q(z) &\ge\, \log|a| - (1 - \beta)\deg Q \int U^{\mu}(z)d\mu_Q(z) \\
&\ge \,(1 - \beta) \left(\log|a| - \deg Q \int U^{\mu}(z)d\mu_Q(z) \right).
\end{align*}
So, if $\int_{\Sigma'} \log|Q| d\mu \ge 0$ for a given nonconstant integer polynomial $Q$, it is also true that
\begin{equation}
\label{eq:sweetened}
\int_{\Sigma'} \log|Q| d\text{sw}(\nu)\ge 0.
\end{equation}

\end{defn}

\begin{rmk}
\label{rmk:unweighted_holder}
If $\Sigma$ is the compact finite union of intervals $[x_1, y_1] \cup \dots \cup [x_k, y_k]$ with 
\[ x_1 < y_1 < \dots < x_k < y_k,\]
then
\[d\mu_{\Sigma}(t) = |P(t)| dt\cdot \prod_{i \le k} |t -x_i|^{-1/2} |t - y_i|^{-1/2}\quad\text{for }\,t \in \Sigma\]
for some polynomial $P$, as follows from \cite[Lemma 2.2]{Pehe90}. In particular, $\mu_{\Sigma}$ is $1/2$-H\"{o}lder in this case, so $\textup{sw}(\nu)$ is H\"{o}lder if $\nu$ is H\"{o}lder.
\end{rmk}

To prove Proposition \ref{prop:limit_Holder}, we will need to generalize Corollary \ref{cor:neg_energy} to non-H\"{o}lder measures.
\begin{prop}
\label{prop:general_energy}
Choose a Borel probability measure $\mu$ whose support is contained in a compact subset of $\R$. Suppose $\int  \log|Q| d\mu$ is nonnegative for every nonzero integer polynomial $Q$. Then 
\[I(\mu)\le 0.\]
\end{prop}
\begin{proof}
Choose $C_0 > 3$ so the support of $\mu$ is contained in $[-C_0+ 1, C_0 - 1]$. Take $\Sigma = [-C_0+1, C_0-1]$ and $\Sigma' = [-C_0, C_0]$. The interval $\Sigma$ has length greater than $4$, so it has capacity greater than $1$.

For a given $\epsilon$ in $(0, 1/2)$, consider the sweetened measure $\text{sw}(\mu * \nu_{\epsilon})$ defined with respect to $\mu$ and $\Sigma$. The measure $\mu * \nu_{\epsilon}$ is H\"{o}lder, so the sweetened measure is also H\"{o}lder. Thus, \eqref{eq:sweetened} implies Corollary \ref{cor:neg_energy} holds for $\text{sw}(\mu * \nu_{\epsilon})$, so $I(\text{sw}(\mu * \nu_{\epsilon})) \le 0$ for all $\epsilon$ in $(0, 1/2)$.

At the same time, the measure of sweetener in $\text{sw}(\mu * \nu_{\epsilon})$ is bounded by $C \epsilon^{1/2}$ for some $C> 0$ not depending on $\epsilon$  by \eqref{eq:good_imitation}. As a result, we have
\[\left|\textup{sw}(\mu * \nu_{\epsilon})(Y) - \mu * \nu_{\epsilon}(Y)\right| \le C\epsilon^{1/2} \]
for any Borel subset $Y$ of $\C$, so $\textup{sw}(\mu * \nu_{\epsilon}) - \mu * \nu_{\epsilon}$ \weaks converges to $0$ as $\epsilon$ tends to $0$. Furthermore, the measures $\mu * \nu_{\epsilon}$ \weaks converge to $\mu$ by Lemma \ref{lem:weaks_smear}.  So the measures $\textup{sw}(\mu * \nu_{\epsilon})$ converge to $\mu$ as $\epsilon$ tends to $0$. Following \cite[Theorem I.6.8]{SaTo97}, the monotone convergence theorem shows $I(\mu) \le 0$.
\end{proof}

\begin{proof}[Proof of Proposition \ref{prop:limit_Holder}]

We may write $\Sigma$ as the union of its set of isolated points with a set of the form $\cup_{k \ge 1} \Sigma_k$, where each $\Sigma_k$ is a compact finite union of intervals and $\Sigma_k \subseteq \Sigma_{k+1}$ for all $k \ge 1$. We will choose $\Sigma_1$ to have capacity $\kappa  > 1$. 

By Proposition \ref{prop:general_energy}, we know that $I(\mu)$ is nonpositive. A countable collection of points has zero capacity; since $I(\mu) > -\infty$, we must have that $\mu(\cup_{k \ge 1} \Sigma_k) = 1$.

For $k \ge 1$, take $\nu_k$ to be the measure defined by $\nu_k(Y) = \mu(Y \cap \Sigma_k)$ for all Borel sets $Y$, and take $\gamma_k = 1 - \nu_k(\Sigma_k)$. Take $\text{sw}(\nu_k)$ to be the sweetened measure of $\nu_k$ defined with respect to $\mu$ and $\Sigma_1$. Supposing $\Sigma$ is contained in $[-C, C]$ for a given $C > 0$, we see that $U^{\nu_k - \mu}$ is at most $\gamma_k \log(2C)$ on $\Sigma$. From the principle of domination \cite[Theorem II.3.2]{SaTo97}, we thus have
\[U^{\nu_k + \gamma_k \mu_{\Sigma_1}}(z) - U^{\mu}(z) \le\gamma_k \log 2C\]
for all $z \in \C$, so the measure of sweetener for $\text{sw}(\nu_k)$ is bounded by $ \gamma_k\log_\kappa (2\kappa C)$. The limit of the $\gamma_k$ is $0$, so it follows that the measures $\text{sw}(\nu_k)$ \weaks converge to $\mu$. Since $\mu$ satisfied the first condition of Theorem \ref{thm:main}, these measures also satisfy this condition by \eqref{eq:sweetened}.

It thus suffices to prove the proposition in the case that $\Sigma$ is a compact finite union of intervals 
\[\Sigma = \bigcup_{i \le n} [x_i, y_i]\quad\text{with }\, x_1 < y_1 < \dots < x_n < y_n.\]
Take $\mu_-$ to be the restriction of $\mu$ to $\bigcup_{i \le n} [x_i, \tfrac{1}{2}(x_i + y_i)]$, and take $\mu_+ = \mu - \mu_-$. Take $\nu_{\epsilon}$ as in Notation \ref{notat:balbump}, and define a measure $\nu^*_{\epsilon}$ on $[-\epsilon, -\epsilon^2]$ by $\nu^*_{\epsilon}(Y) = \nu_{\epsilon}(-Y)$ for every Borel set $Y$. If $2 \epsilon$ is smaller than $y_i - x_i$ for $i \le n$, the measure
\[\mu_{\epsilon} = \mu_- * \nu_{\epsilon} + \mu_+ *\nu_{\epsilon}^*\]
has support contained in $\Sigma$. Applying \eqref{eq:good_imitation} and its analogue for $\nu_{\epsilon}^*$ gives
\[U^{\mu_- *\nu_{\epsilon} }(z) - U^{\mu_-}(z) \le C_1\epsilon^{1/2}\quad\text{and}\quad U^{ \mu_+ * \nu_{\epsilon}^* }(z) - U^{\mu_+}(z) \le C_1\epsilon^{1/2}\]
for all $z \in \C$ and for some $C_1 > 0$ depending just on $\Sigma$. So
\[U^{\mu_{\epsilon}}(z) - U^{\mu}(z) \le 2C_1 \epsilon^{1/2},\] 
so $\text{sw}(\mu_{\epsilon})$ has measure of sweetener bounded by $2C_1 \epsilon^{1/2}/\log \kappa$. As in the proof of Proposition \ref{prop:general_energy}, this gives that $\text{sw}(\mu_{\epsilon}) - \mu_{\epsilon}$ converges to $0$ as $\epsilon$ tends to $0$.

From Lemma \ref{lem:weaks_smear} and its analogue for $\nu_{\epsilon}^*$, we find that $\mu_- *\nu_{\epsilon}$  has limit $\mu_{-}$ and $\mu_+* \nu_{\epsilon}^*$ has limit $\mu_+$ as $\epsilon$ tends to $0$, so $\mu_{\epsilon}$ has \weaks limit $\mu$. So the measures $\text{sw}(\mu_{\epsilon})$ \weaks converge to $\mu$ as $\epsilon$ tends to $0$. These measures are H\"{o}lder by Remark \ref{rmk:unweighted_holder} since $\mu_{\epsilon}$ is H\"{o}lder. Since $\mu$ satisfied condition (1) of Theorem \ref{thm:main}, these measures also satisfy this condition by \eqref{eq:sweetened}. Thus, given any sequence $\epsilon_1 >\epsilon_2 > \dots $ tending to $0$ with $\epsilon_1$ sufficiently small, we see that the measures $\mu_k = \textup{sw}(\mu_{\epsilon_k})$ satisfy the conditions of the proposition.
\end{proof}

\subsection{The limit of Smyth's method}
\begin{notat}
\label{notat:general_optimization}
Take $\Sigma$ to be a closed, possibly unbounded subset of $\R$ containing at most countably many connected components, and take $F: \Sigma \to \R$ to be some continuous function. In the case that $\Sigma$ is unbounded, we assume that, for any sequence $x_1,x_2, \dots $ of points in $\Sigma$ satisfying $\lim_{i \to \infty}|x_i| = \infty$, we have
\begin{equation}
\label{eq:unbounded_restriction}
\lim_{i \to \infty} F(x_i)/\log |x_i| = +\infty.
\end{equation}
We also assume that $\Sigma$ has capacity greater than $1$.

Take $\alpha_1, \alpha_2, \dots$ to be an enumeration of the algebraic integers whose conjugates all lie in $\Sigma$. We take $n_i$ to be the degree of $\alpha_i$, $P_i$ to be its minimal polynomial, and $\alpha_{i1}, \dots, \alpha_{in_i}$ to be its conjugates. The mean value of $F$ on the conjugates of $\alpha_i$ is $  \sum_{j \le n_i}F(\alpha_{ij})/n_i$, and we are interested in the limit point
\[\lambda(\Sigma, F) \,:=\, \liminf_{i \to \infty} \sum_{j \le n_i}  F(\alpha_{ij})/n_i \,=\,\liminf_{i \to \infty} \int_{\Sigma} F d\mu_{P_i}. \]
\end{notat}

\begin{ex}
\label{ex:SSS}
Take $F(x) = x$ and $\Sigma = \R^{\ge 0}$. Then the Schur--Siegel--Smyth trace problem reduces to calculating $\lambda_{\text{SSS}}\,=\,\lambda(\Sigma, F)$.
\end{ex}

With Notation \ref{notat:general_optimization}, we have defined a class of optimization problems where Smyth's approach can be applied \cite{Smyth84a}. As a consequence of Theorem \ref{thm:main}, we can show that the limit of Smyth's method suffices to calculate $\lambda(\Sigma, F)$. In the case that $\Sigma$ is unbounded, this starts with the following result, which makes use of the condition \eqref{eq:unbounded_restriction}.
\begin{lem}
\label{lem:secret_bound}
Given $\Sigma$ and $F$ as in Notation \ref{notat:general_optimization}, there is a compact subset $\Sigma_0$ of $\Sigma$ so, for any probability measure $\mu$ with compact support contained in $\Sigma$ but not contained in $\Sigma_0$, there is a probability measure $\mu'$ with support contained in $\Sigma_0$ such that $\int F d\mu > \int F d\mu'$ and such that $\int \log|Q|d\mu' \ge 0$ for every integer polynomial $Q$ for which $\int \log|Q| d\mu \ge 0$.
\end{lem}
\begin{proof}
Take $\Sigma_1$ to be a compact finite union of intervals contained in $\Sigma$ of capacity $\kappa > 1$, and choose $L > 1$ so $\Sigma_1$ is contained in $[-L, L]$. Take $\mu_1$ to be the unweighted equilibrium measure for $\Sigma_1$. Then we have
\[U^{\mu_1}(z) \le -\log \kappa \quad\text{for all } \,z \in \C\]
by \cite[(I.1.4)]{SaTo97}. Take $R = L(\kappa - 1)^{-1}$; this is chosen so $\log (R + L) - \log R$ equals $\log \kappa$. Then, given any probability measure $\nu$ with compact support contained in $\R \backslash(-R, R)$, we have
\[U^{\nu}(x) = \int -\log|x - t| d\nu(t) \ge \int -\log(|t| + L)d\nu(t)  \ge -\log \kappa - \int \log|t| d\nu(t)\]
for all $x \in \Sigma_1$, with the final inequality following from the fact that $\log(|t| + L) - \log|t|$ attains its maximum outside $(-R, R)$ for $t = - R, R$. So
\begin{equation}
\label{eq:intlogt}
U^{\mu_1}(z) - U^{\nu}(z) \le \int \log|t| d\nu(t)
\end{equation}
for all $z \in \Sigma_1$. By the principle of domination \cite[Theorem II.3.2]{SaTo97}, this inequality holds for all $z \in \C$.

Given any $x \in \Sigma$, we know that $F(z) \ge F(x)$ for $z$ outside some compact subset of $\Sigma$ by the growth condition on $F$. Since $F$ is continuous, it must attain its minimum, and we take $c_0 = \min(0, \min_{x \in \Sigma} F(x))$. We also take $c_1 = \max\left(0, \int F d\mu_1\right)$. Choose $M > \max(\kappa, L, R)$ such that
\begin{equation}
\label{eq:infinitymorelikefinity}
\frac{F(x)}{\log|x|} > \frac{2(c_1 - c_0)}{\log \kappa} \quad\text{ for } x \in \Sigma \backslash [-M, M].
\end{equation}
We claim that $\Sigma_0 = [-M, M] \cap \Sigma$ satisfies the conditions of the proposition. So, given $\mu$ as in the statement of the lemma, take $\nu$ to be the restriction of $\mu$ to $[-M, M]$, and consider the measure $\textup{sw}(\nu)$ defined with respect to $\mu$ and $\Sigma_1$. In this case, the quantity \eqref{eq:def_B} satisfies
\begin{equation}
\label{eq:Blicious}
B \le \int \log|t| d(\mu - \nu)(t)
\end{equation}
by \eqref{eq:intlogt} since $\mu - \nu$ is supported outside $(-R, R)$. Defining $\beta$ and $\gamma$ as in Definition \ref{defn:sweet}, we have
\[\int F d\textup{sw}(\nu) = (\beta + \gamma - \beta \gamma)\int F  d\mu_1 + (1 - \beta)\int F d\nu.\]
We first note that
\[(\beta + \gamma - \beta \gamma)\int F  d\mu_1 \le (\beta + \gamma - \beta \gamma)c_1 \le \left(\frac{B}{\log \kappa} + \gamma\right) c_1,\]
as $\beta - \beta\gamma$ is at most $\beta$, which is at most $B/\log \kappa$. We also have
\begin{align*}
(1 - \beta)\int F d\nu =&   - \beta \int F d\nu - \int F d(\mu - \nu) +\int F d\mu\\
 \le& -\frac{B}{\log \kappa} c_0  - \int F d(\mu - \nu) + \int F d\mu.
\end{align*}
But
\begin{alignat*}{2}
-\int F d(\mu - \nu)  &\,<\, -\frac{2(c_1 - c_0)}{\log \kappa} \int \log|t| d(\mu - \nu)(t) \qquad&&\text{by \eqref{eq:infinitymorelikefinity}}\\
&\,\le\, -\frac{2(c_1 - c_0)}{\log \kappa} \max(B, \gamma\log M)&&\text{by \eqref{eq:Blicious}},
\end{alignat*}
and these inequalities together imply
\[\int  F d\textup{sw}(\nu)  <  \int F d\mu.\]
Taking $\mu' = \textup{sw}(\nu)$, the lemma follows from \eqref{eq:sweetened}.
\end{proof}

\begin{thm}
\label{thm:general_Smyth}
Take $\Sigma$ and $F$ as in Notation \ref{notat:general_optimization}. Take $\lambda_{\textup{Smyth}}(\Sigma, F)$ to be the least upper bound of  the real $\lambda$  for which there is a finite sequence $Q_1, \dots, Q_N$ of nonzero integer polynomials and a finite sequence $a_1, \dots, a_N$ of positive numbers such that
\begin{equation}
\label{eq:gen_Smyth}
F(x) \ge \lambda + \sum_{i =1}^N a_i \log|Q_i(x)|\quad\text{for all } \,x \in \Sigma.
\end{equation}
Then $\lambda(\Sigma, F) = \lambda_{\textup{Smyth}}(\Sigma, F)$.
\end{thm}
\begin{proof}
We first will show that $\lambda_{\textup{Smyth}}(\Sigma, F)\le \lambda(\Sigma, F)$. To see this, take $P$ to be any nonconstant monic integer polynomial, and suppose \eqref{eq:gen_Smyth} holds for a given $\lambda$ and $a_1, \dots, a_N$. Following the argument appearing after Theorem \ref{thm:Serre}, we see that
\[\int F d\mu_P \ge \lambda + \frac{1}{\deg P}\sum_{i = 1}^N a_i\log|\text{res}(P, Q_i)|,\]
and the right hand side of this equation is at least $\lambda$ unless some root of $P$ is a root of some $Q_i$. We may conclude that $\lambda \le \lambda(\Sigma, F)$, and hence that $\lambda_{\textup{Smyth}}(\Sigma, F)\le \lambda(\Sigma, F)$.

We now need to show that $\lambda(\Sigma, F)$ is at most $\lambda_{\textup{Smyth}}(\Sigma, F)$. Take $Q_1, Q_2, \dots$ to be an enumeration of the nonconstant integer polynomials, and take $\Sigma_0$ to be a compact subset of $\Sigma$ of capacity greater than $1$ obeying the conditions of Lemma \ref{lem:secret_bound}. Given a positive integer $N$, take $\mathscr{M}_N$ to be the set of probability measures $\mu$ with compact support contained in $\Sigma$ satisfying $\int \log|Q_i| d\mu > -\infty$ for $i \le N$, and consider the convex subset
\[V_N = \left\{\left(\int F d\mu, \int \log|Q_1| d\mu, \dots , \int \log|Q_N| d\mu\right)\,:\,\, \mu \in \mathscr{M}_N\right\} \subseteq \R^{N + 1}.\]
Given $\lambda \in \R$, we also define the convex set
\[W_{\lambda, N} = \{(x_0, \dots, x_{N}) \in\R^{N+1} \,:\,\, x_0 \le \lambda \text{ and } x_i \ge 0 \text{ for } 1 \le i \le N\},\]
and we take $W_{\lambda, N}^{\circ}$ to be the interior of this set.

As $\inf_{x\in \Sigma} F(x) > -\infty$, we see that $V_N$ is disjoint from $W_{\lambda, N}$ for some sufficiently negative $\lambda$. At the same time, taking $\mu_{\Sigma_0}$ to be the unweighted equilibrium measure of $\Sigma_0$, we have $U^{\mu_{\Sigma_0}}(z) \le I(\mu_{\Sigma_0}) < 0$ for all $z \in \C$ by \cite[(I.1.4)]{SaTo97}, so \eqref{eq:logQ_pot} gives  $\int\log|Q_i| d\mu_{\Sigma_0} >  0$ for all $i$. Taking $c_1 = \int F d\mu_{\Sigma_0}$, we see that $W_{c_1 + 1, N}^{\circ}$ meets $V_N$.

Take $\lambda_N$ to be the least upper bound of the $\lambda$ for which $V_N$ and $W_{\lambda, N}^{\circ}$ are disjoint. Then $V_N$ and $W_{\lambda_N - N^{-1}, N}^{\circ}$ are disjoint. By the hyperplane separation theorem \cite[Theorem 4e]{Klee68}, there is then a nonzero element $\mathbf{a}$ in $\R^{N+1}$ for which
\begin{equation}
\label{eq:avaw}
\mathbf{a} \cdot \vv < \mathbf{a} \cdot \ww \quad\text{for all }\, \vv \in V_N \text{ and } \ww \in W_{\lambda_N - N^{-1}, N}^{\circ}.
\end{equation}
Writing $\mathbf{a}$ in the form $(a_0, \dots, a_N)$, we see that $a_1, \dots, a_N$ must all be nonnegative, since otherwise the right hand side of this inequality can be made arbitrarily negative. Taking $\vv$ to be a vector  in $W_{c_1 + 1, N}^{\circ} \cap V_N$, we must have $\mathbf{a} \cdot \vv < \mathbf{a} \cdot \ww$ for all $\ww$ in $W_{ \lambda_N - N^{-1}, N}$. This implies that $a_0$ is negative; we may assume it is $-1$.

Then \eqref{eq:avaw} takes the form
\[\int F d\mu > \lambda_N - N^{-1} + \sum_{i \le N} a_i  \int \log|Q_i| d\mu - \epsilon\]
for all $\mu \in \mathscr{M}_N$ and any $\epsilon > 0$. Applying this to the Dirac measure $\delta_x$ and letting $\epsilon$ tend to $0$ gives
\[F(x) \ge \lambda_N - N^{-1} + \sum_{i \le N} a_i \log|Q_i(x)| \quad\text{for all }\, x \in \Sigma\]
unless $x$ is a root of some $Q_i$ with $i \le N$. For such an $x$, the inequality follows since the left hand side is finite and the right is $-\infty$.

So
\[\lambda_{\textup{Smyth}}(\Sigma, F) \ge \limsup_{N \to \infty} (\lambda_N - N^{-1}) =  \limsup_{N \to \infty} \lambda_N. \]

In addition, since $W_{\lambda_N + N^{-1}, N}$ meets $V_N$, there is a measure $\mu_N \in \mathscr{M}_N$ so $\int \log|Q_i| d\mu_N \ge 0$ for $i \le N$ and $\int F d\mu_N \le \lambda_N + N^{-1}$. Using Lemma \ref{lem:secret_bound}, we will assume that this measure has support contained in $\Sigma_0$ for every $N$. By Helly's selection theorem \cite[Theorem 0.1.3]{SaTo97}, some subsequence of the $\mu_N$ converge; take $\mu$ to be their limit. Given a nonzero integer polynomial $Q$, $\log|Q|$ is upper semicontinuous, so we may apply the monotone convergence theorem as in \cite[(0.1.2)]{SaTo97} to conclude
\[\int \log|Q|d\mu \ge \liminf_{N \to \infty}\int \log|Q| d\mu_N \ge 0.\]
The definition of \weaks convergence gives
\[\int F d\mu \le \limsup_{N \to \infty}\int Fd\mu_N \le \limsup_{N \to \infty} (\lambda_N + N^{-1}) = \limsup_{N \to \infty} \lambda_N.\]
 So $\mu$ is a probability measure with compact support contained in $\Sigma$ satisfying the first condition of Theorem \ref{thm:main},  and it satisfies $\int F d\mu \le \lambda_{\textup{Smyth}}(\Sigma, F)$. The theorem follows from Theorem \ref{thm:main}.
\end{proof}

The framework of Notation \ref{notat:general_optimization} can be used to study abelian varieties over finite fields through the following proposition.
\begin{prop}
\label{prop:Honda}
Choose a prime power $q$. Take $\Sigma$ to be the interval $\left[-2\sqrt{q},\, 2\sqrt{q}\right]$ and define $F: \Sigma \to \R$ by $F(x) = \log|q + 1- x|$. Then, for any $\epsilon > 0$, there are infinitely many $\FFF_q$-simple abelian varieties $A$ over $\FFF_q$ satisfying
\[\big(\# A(\FFF_q)\big)^{1/\dim A} \,\le\, \exp\big(\lambda( \Sigma, F) + \epsilon\big),\]
infinitely many more satisfying
\[\big(\# A(\FFF_q)\big)^{1/\dim A} \,\ge\, \exp\big(-\lambda( \Sigma, -F) -  \epsilon\big),\]
but only finitely many satisfying
\[\big(\# A(\FFF_q)\big)^{1/\dim A} \,\not\in \, \big[\exp\big(\lambda( \Sigma, F) - \epsilon\big),\,\, \exp\big(-\lambda( \Sigma, -F) +  \epsilon\big)\big].\]
\end{prop}
\begin{proof}
This follows from Honda--Tate theory \cite{Honda68}; see \cite[Proposition 2.1]{Kade21}.
\end{proof}

In the case that $q$ is a large square, the optimization problem of Proposition \ref{prop:Honda} approximately reduces to the trace problem. Specifically, we have the following result, whose argument follows the proof of \cite[Proposition 2.4]{Kade21}.
%Kadets is the opposite inequality, same idea.
\begin{prop}
\label{prop:Kade}
Define $ \lambda_{\textup{SSS}}$ as in Example \ref{ex:SSS}. Given any $\epsilon > 0$, there is $C > 0$ so that, for any square prime power $q$ satisfying $q > C$, there are infinitely many $\FFF_q$-simple abelian varieties satisfying
\[\#A(\FFF_q) \ge (q + 2\sqrt{q} +1 - \lambda_{\textup{SSS}} - \epsilon)^{\dim A}\]
and infinitely many more satisfying
\[\#A(\FFF_q) \le (q - 2\sqrt{q} + 1 + \lambda_{\textup{SSS}} + \epsilon)^{\dim A}.\]
\end{prop}
\begin{proof}
From the proof of Theorem \ref{thm:general_Smyth}, we see that there is a measure $\mu$ and real number $C_0 > 4$ so $\mu$ has support contained in $[0, C_0]$, so $\mu$ satisfies the criteria of Theorem \ref{thm:main}, and so $\int x d\mu(x) = \lambda_{\textup{SSS}}$. Suppose $q$ is a square prime power satisfying $4 \sqrt{q} \ge C_0$,
and define measures $\mu^+_q$ and $\mu^-_q$ by
\[\mu^+_q(Y) = \mu(Y + 2\sqrt{q}) \quad\text{and}\quad \mu^-_q(Y) = \mu(-Y + 2\sqrt{q})\]
for every Borel subset $Y$ of $\R$. These measures both have support in $[-2\sqrt{q}, 2\sqrt{q}]$. Since $2\sqrt{q}$ is a rational integer, these measures satisfy condition (1) of Theorem \ref{thm:main} and hence satisfy condition (2) with $\Sigma = [-2\sqrt{q}, 2\sqrt{q}]$.

We have
\[\int \log|q + 1 -x| d\mu_q^-(x) = \int \log \left|q - 2\sqrt{q} + 1 + x\right| d\mu(x).\]
Taking a Taylor expansion, we also have 
\[\log\left|q - 2\sqrt{q} + 1 + x\right| \,=\, \log| q - 2\sqrt{q} + 1 + \lambda_{\text{SSS}}| + \frac{ x- \lambda_{\text{SSS}}}{q - 2\sqrt{q} + 1 + \lambda_{\text{SSS}}} + O(q^{-2}),\]  
for sufficiently large $q$ and all $x$ in $[0, C_0]$, where the implicit constant depends just on $C_0$. Integrating against $\mu$ gives
\[\int \log|q + 1 -x| d\mu_q^-(x)  \,=\, \log| q - 2\sqrt{q} + 1 + \lambda_{\text{SSS}}| + O(q^{-2})\]
for sufficiently large $q$, giving the result for small point counts by Proposition \ref{prop:Honda}. An analogous argument for $\mu^+_{q}$ gives the result for large point counts.
\end{proof}

\begin{rmk}
Recall that, in the course of the proof of Theorem \ref{thm:main}, we have imposed restrictions on the coefficients of the polynomials $R_k$ modulo $4$. The same argument allows us to restrict these coefficients to be in certain congruence classes modulo $4q$. Following the argument of \cite{BCLPS21}, this freedom can be used to force the infinite families of abelian varieties constructed in Proposition \ref{prop:Honda} to all be geometrically simple and ordinary.
\end{rmk}

\subsection{Serre's example}
\label{ssec:Serre}
A natural method to verify that a measure satisfies the condition of Theorem \ref{thm:main} is to use the following proposition, which again takes advantage of the integrality of the resultant of integer polynomials.
\begin{prop}
\label{prop:need_bal}
Take $\Sigma$ to be a compact subset of $\R$ of capacity greater than $1$ with at most countably many connected components, and take $\mu$ to be a probability measure on $\Sigma$. Suppose there is a finite sequence $Q_1, \dots, Q_N$ of distinct irreducible primitive integer polynomials and a finite sequence $a_1, \dots, a_N$ of positive numbers such that $\sum_{i \le N} a_i \deg Q_i \le 1$ and \begin{equation}
\label{eq:need_bal}
U^{\mu}(z) \le \sum_{i \le N} -a_i \log|Q_i(z)| \quad\text{for all } z \in \C.
\end{equation}
Suppose further that $\int_{\Sigma} \log|Q_i| d\mu$ is nonnegative for $i \le N$.

Then $\mu$ obeys the equivalent conditions of Theorem \ref{thm:main}.
\end{prop}
\begin{proof}
Given a nonzero integer polynomial $Q$, we need to show
\[\int_{\Sigma} \log|Q| d\mu \ge 0.\]
We may assume that $Q$ is irreducible and of positive degree $n$ without loss of generality. If it is a multiple of some $Q_i$, the inequality follows by assumption. Otherwise, take $c$ to be the leading term of $Q$. We have
\begin{align*}
\int_{\Sigma} \log|Q| d\mu \,&=\, \log|c| - n\int U^{\mu} d\mu_Q\\
&\ge\, \sum_{i \le N} a_i\left(\deg Q_i \cdot\log|c| +n \int \log|Q_i| d\mu_Q\right) \ge 0,
\end{align*}
with the final inequality following since the resultants $\res(Q, Q_i)$ are nonzero rational integers. 
\end{proof}

\begin{ex}
\label{ex:Serre}
Take $a = .087353$, $b = 4.411076$, and $\gamma = .215485$. We consider the probability measure 
\[\mu = \gamma \nu_{[a, b]} + (1 - \gamma) \mu_{[a, b]},\]
where $\nu_{[a, b]}$ is defined as in Notation \ref{notat:balbump} and $\mu_{[a, b]}$ is the unweighted equilibrium measure on $[a, b]$. The parameters $a$ and $b$  are rounded forms of the ones considered in \cite{Serre98b}.

From \eqref{eq:first_balay} and \cite[(8) in Appendix B]{AgPe08}, we know that
\[U^{\mu}(z) + \gamma \log|z| = \begin{cases} C &\text{ if } z \in [a, b] \\ \le C&\text{ otherwise}\end{cases}\]
with
\[ C = - \log\left(\frac{b - a}{4}\right) +  \gamma \log\left(\frac{a + 2\sqrt{ab} + b}{4}\right) \approx -1.3 \cdot10^{-7}.\]

From Lemma \ref{lem:balay} (3) and \cite[(10) in Appendix B]{AgPe08}, we also have
\[\int \log|x| d\mu(x) = -\gamma \log\left(\frac{a + 2\sqrt{ab} + b}{4ab}\right)  + (1 - \gamma)  \log\left(\frac{a + 2\sqrt{ab} + b}{4}\right) \approx 1.5 \cdot10^{-6}.\]
Applying Proposition \ref{prop:need_bal} shows that the equivalent conditions of Theorem \ref{thm:main} hold for this measure.

We have $\int_a^b \frac{dx}{\pi \sqrt{(b - x)(x - a)}} =1$; indeed, this integrand gives the form of the unweighted equilibrium measure on $[a, b]$ \cite[Example I.3.5]{SaTo97}. Applying this identity for $\nu_{[a, b]}$ and \cite[(11) in Appendix B]{AgPe08} for $\mu_{[a, b]}$, we have
\[\int x d\mu(x)\, =\, \gamma \sqrt{ab} + (1 - \gamma) \frac{a+b}{2} \,\in\, (1.898303, 1.898304).\]
Theorem \ref{thm:Serre} then follows from applying Theorem \ref{thm:main} to $\mu$.

We have constructed this example from an argument of Serre \cite[Appendix B]{AgPe08} in a letter to Smyth. In this letter Serre showed that, for any $t > 0$ and $\gamma \in [0, 1]$, the inequality
\[x \ge c + t \gamma \log(x) + (1 - \gamma) t \tfrac{1}{\deg\, R} \log|R(x)|\]
can only hold for all $x > 0$ and real polynomials $R$ with leading term and constant term of modulus at least $1$ if $c \le 1.89830\dots$. He did this by integrating  the sides of this inequality first with respect to the measure $\mu_{[a, b]}$, and then with respect to the pushforward of the equilibrium measure on $[b^{-1}, a^{-1}]$ under the map $z \mapsto 1/z$. This latter measure is $\nu_{[a, b]}$, and the convex combination $\mu$ considered above repackages this two-part argument in a single measure.

The choices here are nearly optimal. Specifically, if a measure $\mu$ satisfies the condition of Proposition \ref{prop:need_bal} with  $\Sigma = \R^{\ge 0}$, $N =1$, and $Q_1(z) =z$, then $\int xd\mu(x) \ge 1.898302\dots$ \cite[p. 10]{AgPe08}. But that does not mean that this example is nearly optimal for the trace problem. After all, we have the following proposition.
\end{ex}

\begin{prop}
\label{prop:crater}
Take $\Sigma$ and $F$ as in Notation \ref{notat:general_optimization}. Choose a probability measure $\mu$ with compact support contained in $\Sigma$ such that $\int \log|Q| d\mu \ge 0$ for every nonzero integer polynomial $Q$ and $\int F d\mu = \lambda(\Sigma, F)$, and take $P$ to be an irreducible monic nonconstant integer polynomial satisfying
\[\int_{\Sigma} F d\mu_P < \lambda(\Sigma, F)\]
whose roots are contained in an open subset of $\R$ contained in $\Sigma$.

Then $\int_{\Sigma} \log|P| d\mu = 0$, and the support of $\mu_P$ is disjoint from the support of $\mu$.
\end{prop}
\begin{proof}
Using Lemma \ref{lem:secret_bound}, we will assume without loss of generality that $\Sigma$ is compact. Take $\mathscr{M}$ to be the set of probability measures $\nu$ with support contained in $\Sigma$ that satisfy $\int \log|Q| d\nu \ge 0$ for all integer polynomials $Q$ that are indivisible by $P$.  The counting measure $\mu_P$ is in $\mathscr{M}$, as $\int \log|Q|d\mu_P$ equals $\log |\res(Q, P)|$, which is nonnegative for any such $Q$.

Take $\Sigma_0$ to be a compact subset of $\Sigma$ of capacity $\kappa > 1$ which does not contain any root of $P$. Define $\nu_{\epsilon}$ as in Notation \ref{notat:balbump}. Since the roots of $P$ are contained in the interior of $\Sigma$, we find that the sweetened measure $\text{sw}(\mu_P *\nu_{\epsilon})$ defined with respect to $\mu_P$ and $\Sigma_0$ has support contained in $\Sigma$ for sufficiently small $\epsilon$, so $\text{sw}(\mu_P *\nu_{\epsilon})$  lies in $\mathscr{M}$ by \eqref{eq:sweetened}. As in the proof of Proposition \ref{prop:general_energy}, the measures $\text{sw}(\mu_P* \nu_{\epsilon})$ \weaks converge to $\mu_P$ as $\epsilon$ tends to $0$, so we have
\[\int F d\text{sw}(\mu_P *\nu_{\epsilon}) < \lambda(\Sigma, F)\]
 for all sufficiently small $\epsilon$. Note also that $\int \log|P| d\text{sw}(\mu_P *\nu_{\epsilon})$ is finite. So, if $\int \log|P| d\mu > 0$, there would be some convex combination  $\nu = \gamma \text{sw}(\mu_P *\nu_{\epsilon}) + (1 - \gamma) \mu$  for some choice of real parameters $\gamma, \epsilon \in (0, 1/2)$ such that $\nu$ satisfies the conditions of Theorem \ref{thm:main} and $\int F d\nu < \lambda(\Sigma, F)$. This cannot happen, so we can conclude that $\int \log|P| d\mu = 0$. Considering the convex planar region
\[\left\{ \left(\int F d\nu, \int \log|P| d\nu\right)\,:\,\, \nu \in \mathscr{M}\text{ and } \int \log|P| d\nu > -\infty\right\} \subseteq \R^2,\]
we see that $(\lambda(\Sigma, F) -  \epsilon_1, \epsilon_2)$ is outside this region for any $\epsilon_1, \epsilon_2 \ge 0$ with $\epsilon_1 + \epsilon_2 > 0$. It follows from the standard hyperplane separation theorem \cite[p. 133]{Klee68} that, for some $a \in [0, 1]$, the infimum
\begin{equation}
\label{eq:mu_winner}
\inf_{\nu \in \mathscr{M}}\left( (1 -a) \cdot \int F d\nu \, -\, a \cdot\int \log|P| d\nu\right)
\end{equation}
is attained at $\nu = \mu$. By  considering $\textup{sw}(\mu_P * \nu_{\epsilon})$, we see that $a$ must be positive.

Choose $\epsilon > 0$, take $\mu_{\epsilon}$ to be the restriction of $\mu$ to
\[\big\{x \in \Sigma\,:\,\, |x - \alpha| \ge \epsilon \text{ for each root } \alpha \text{ of } P\big\},\] 
and take $\gamma = 1 - \mu_{\epsilon}(\Sigma)$. We consider the sweetened measures $\textup{sw}(\mu_{\epsilon})$ defined with respect to $\mu$ and $\Sigma_0$, and we define $ \beta$ from this measure as in Definition \ref{defn:sweet}. Applying the principle of domination as in the proof of Proposition \ref{prop:limit_Holder} gives $\beta \ll \gamma$, where the implicit constants here and for the rest of the proof depend just on $\Sigma$, $\Sigma_0$, $P$, and $F$.

Since $P$ is monic and has all its roots in $\Sigma$, and since the support of $\mu - \mu_{\epsilon}$ is contained in the union of intervals of radius $\epsilon$ centered at each root of $P$, we have
\[- \gamma \log \epsilon + \int \log|P| d(\mu - \mu_{\epsilon}) \, \ll\, \gamma\quad\text{and}\quad - \gamma \log \epsilon + \int \log|P| d\mu \, \ll\, 1.\]
Since we have $\beta \ll \gamma$, a convex combination of these inequalities gives
\[- \gamma\log\epsilon  + \int \log|P| d\mu  - (1 - \beta) \int\log |P| d\mu_{\epsilon} \,\ll \,\gamma.\]
Taking $\mu_{\Sigma_0}$ to be the unweighted equilibrium measure on $\Sigma_0$, we have $\int \log|P| d\mu_{\Sigma_0} \ge 0$ since $\Sigma_0$ had capacity greater than $1$. So we are left with
\[- \gamma  \log \epsilon + \int \log|P| d(\mu - \text{sw}(\mu_{\epsilon}))  \ll \gamma.\]
We also have
\[\int Fd(\mu - \text{sw}(\mu_{\epsilon})) \gg -\gamma\]
by the bounds on $\beta$ since $F$ is continuous. So
\[\left( (1 -a) \cdot \int F d(\mu -\text{sw}(\mu_{\epsilon}))\, -\, a \cdot\int \log|P| d(\mu -\text{sw}(\mu_{\epsilon}))\right)  + a\gamma \log \epsilon \,\gg\, -\gamma.\]

But $\text{sw}(\mu_{\epsilon})$ lies in $\mathscr{M}$ by \eqref{eq:sweetened}. Since \eqref{eq:mu_winner} is minimized at $\mu$ in $\mathscr{M}$, we have
\[\left( (1 -a) \cdot \int F d(\mu -\text{sw}(\mu_{\epsilon}))\, -\, a \cdot\int \log|P| d(\mu -\text{sw}(\mu_{\epsilon}))\right) \le 0.\]
We are left with $a \gamma \log \epsilon \gg -\gamma$, so $\gamma$ is $0$ if $\epsilon$ is sufficiently small.  This implies that none of the roots of $P$ are in the support of $\mu$.
\end{proof}
This proposition shows that the measure of Example \ref{ex:Serre} does not give the ideal bound on the trace problem. Indeed, the $14$ totally positive algebraic integers $\alpha$ with $\tr(\alpha)/\deg(\alpha)$ at most  $1.793$ all lie in the support of $\mu$ \cite{WaWuWu21}, while the support of an optimal measure would contain none of these points by Proposition \ref{prop:crater}.

To make progress past this point on upper bounds for the trace problem, the natural next step is to apply Proposition \ref{prop:need_bal} with more than one auxiliary polynomial $Q_i$. This can be done; the technique of balayage appearing in e.g. \cite[Theorem II.4.4]{SaTo97} gives a natural way of producing measures satisfying the condition of Proposition \ref{prop:need_bal}, with Example \ref{ex:Serre} showing this in one special case. However, as the number of polynomials increases, the complexity of the corresponding optimization problem increases. We hope to return to these more complex optimization problems in future work.

\bibliography{references}{}
\bibliographystyle{amsplain}

\end{document}